\documentclass[12pt,centertags]{amsart}
\usepackage{amssymb}
\usepackage{amsmath,amstext,amsthm,a4,amssymb,amscd}
\usepackage[mathscr]{eucal}
\usepackage{mathrsfs}
\usepackage{epsf}
\numberwithin{equation}{section}
\allowdisplaybreaks[3]

 \DeclareMathOperator{\Ker}{Ker}
 \DeclareMathOperator{\Dom}{Dom}

\newtheorem{thm}{Theorem}[section]
\newtheorem{lemma}[thm]{Lemma}
\newtheorem{prop}[thm]{Proposition}
\newtheorem{cor}[thm]{Corollary}
\theoremstyle{definition}
\newtheorem{rem}[thm]{Remark}
\theoremstyle{definition}
\newtheorem{defn}[thm]{Definition}
\newcommand{\be}{\begin{eqnarray}}
\newcommand{\ee}{\end{eqnarray}}

\newcommand{\comment}[1]{}
\newcommand{\sbt}{\begin{picture}(-1,1)(-1,-3)\circle*{3}\end{picture}}

\begin{document}


\title[The second coefficient of the asymptotic
expansion]{The second coefficient of the asymptotic \\
expansion of the Bergman kernel\\
of the Hodge-Dolbeault operator}
\author{Wen Lu}

\address{Mathematisches Institut,
Universit\"at zu K\"oln, Weyertal 86-90, 50931 K\"oln, Germany}
\email{wlu@math.uni-koeln.de}
\date{\today}

\begin{abstract}
We calculate the second coefficient of the asymptotic expansion of
the Bergman kernel of the Hodge-Dolbeault operator associated to high powers of
a Hermitian line bundle with non-degenerate curvature,
using the method of formal power series developed by Ma and Marinescu.
\end{abstract}

\maketitle

\setcounter{section}{-1}

\section{Introduction} \label{s0}
The study of the asymptotic expansion of Bergman kernels has attracted much attention recently.
The existence of the diagonal asymptotic expansion of the Bergman kernel of high
tensor powers of a positive line bundle over a compact complex manifold was first established
by Tian \cite{Tian90}, Ruan \cite{Ruan98}, Catlin \cite{Catlin} and Zelditch \cite{Zelditch98}.
Tian \cite{Tian90},  followed by
Lu \cite{Lu00} and Wang \cite{Wang05}, derived explicit formulae for
several terms of the asymptotic expansion on the diagonal, via Tian's method of peak sections.

Using Bismut-Lebeau's analytic localization techniques, Dai, Liu and Ma \cite{Dai04} established the
full off-diagonal asymptotic expansion of the Bergman kernel of the Spin$^{c}$ Dirac operator associated
to high powers of a Hermitian line bundle with positive curvature in the general context of symplectic manifolds.
Moreover, they calculated the second coefficient of the expansion in the case of K\"ahler
manifolds. Later, Ma and Marinescu \cite{Ma08} studied the expansion of generalized Bergman kernels
associated to Bochner-Laplacians and developed a method of formal power series to compute the coefficients.
By the same method, Ma and Marinescu \cite[Th.\,2.1]{Ma06} computed the second coefficient
of the asymptotic expansion of the Bergman kernel of the Spin$^{c}$ Dirac operator
acting on high tensor powers of line bundles with positive curvature in the case of symplectic manifolds.
In the same vein, they computed in \cite{Ma12} the first three coefficients of the expansion
of the kernel of Toeplitz operators. We recommend the survey \cite{Ma10} for the expansion
of the kernel of Toeplitz operators in the context of geometric quantization.

In this paper we consider the Hodge-Dolbeault operator (which is a modified Dirac operator (see (\ref{1.3})))
associated to high powers of a Hermitian line bundle with non-degenerate curvature over a
compact complex manifold. For the non-degenerate curvature case, Ma and Marinescu \cite{Ma06}
obtained the expansion of the Bergman kernel of the
Spin$^{c}$ Dirac operator \cite[Th.\,1.7]{Ma06} on any symplectic manifold and
they computed the first two coefficients \cite[Th.\,2.1]{Ma06}
in the case of positive curvature.
Berman and Sj\"ostrand \cite{Berman} also studied the asymptotic expansion for Bergman
kernels for high powers of holomorphic line bundles with non-degenerate curvature over compact
complex manifolds.

This paper is a continuation of \cite{Ma06}. We compute the second coefficient of asymptotic
expansion of the Hodge-Dolbeault operator by means of the method in \cite{Ma06, Ma08}.
Compared to \cite{Ma06, Ma08}, the main feature in this paper
is that we perform our calculations for line bundles with non-degenerate curvature of arbitrary
signature.

Let $(X,J)$ be a compact connected complex manifold with complex structure $J$ and $\dim_{\mathbb{C}}X=n$.
Let $(L, h^{L})$ be a
holomorphic Hermitian line bundle on $X$, and let $\nabla^{L}$ be the Chern connection of $(L, h^{L})$ with the curvature  $R^{L}=(\nabla^{L})^{2}$.

{\bf {Our basic assumption}} is that $\omega:=\frac{\sqrt{-1}}{2\pi}R^{L}$ defines a symplectic form on
$X$.

The complex structure $J$ induces a splitting
$TX\otimes_{\mathbb{R}}\mathbb{C}=T^{(1,0)}X\oplus T^{(0,1)}X$,
where $T^{(1,0)}X$ and $T^{(0, 1)}X$ are the eigenbundles of $J$
corresponding to the eigenvalues $\sqrt{-1}$ and $-\sqrt{-1}$
respectively. Since the ${J}$-invariant bilinear form $\omega(\cdot, J\cdot)$ is non-degenerate on $TX$,
there exist $J$-invariant subbundles denoted $V, V^{\bot}\subset TX$ such that
\begin{align}
\omega(\cdot, J\cdot)\big|_{V}<0,\ \ \omega(\cdot, J\cdot)\big|_{V^{\bot}}>0
\end{align}
and $V, V^{\bot}$ are orthogonal with respect to $\omega(\cdot, J\cdot)$.
Equivalently, there exist subbundles $W, W^{\bot}\subset T^{(1, 0)}X$ such that
$W\oplus W^{\bot}=T^{(1, 0)}X$, $W, W^{\bot}$ orthogonal with respect to $\omega$ and
\begin{align}
R^{L}(u, \overline{u})<0,\ \textup{for}\ u\in W; \ R^{L}(u, \overline{u})>0,\ \textup{for}\ u\in W^{\bot}.
\end{align}
Set $\textup{rank} W=q$. Then the curvature $R^{L}$ is non-degenerate of signature $(q, n-q)$.
Now take the Riemannian metric $g^{TX}$ on $X$ to be
\begin{align}\label{1.12a}
g^{TX}:=-\omega(\cdot, J\cdot)\big|_{V}
\oplus\omega(\cdot, J\cdot)\big|_{V^{\bot}}.
\end{align}
Since $\omega$ is compatible with the complex structure $J$, then the metric $g^{TX}$ is
also compatible with $J$. Note that $(X, g^{TX})$ is not necessarily K\"ahler. Denote
by $\big\langle\cdot,\cdot\big\rangle$ the $\mathbb{C}$-bilinear form on $TX\otimes_{\mathbb{R}}\mathbb{C}$
induced by $g^{TX}$.

Let $(E, h^{E})$ be a holomorphic Hermitian
vector bundle on $X$, and let $\nabla^{E}$ be the Chern connection of $(E, h^{E})$ with curvature $R^{E}=(\nabla^{E})^{2}$.
Denote by $\Omega^{0, j}(X, L^{p}\otimes E)$ the space of smooth
$(0,j)$-forms over $X$ with values in $L^{p}\otimes E$ and set $\Omega^{0,\bullet}(X, L^{p}\otimes E)=\oplus^{n}_{j=0}
\Omega^{0, j}(X, L^{p}\otimes E)$. We still denote by $\big\langle\cdot, \cdot\big\rangle$ be the fibrewise metric
on $\Lambda(T^{\ast(0, 1)}X)\otimes L^{p}\otimes E$ induced by $g^{TX}, h^{L}$ and $h^{E}$.
Let $dv_{X}$ be the Riemannian volume form of
$(X, g^{TX})$. The $L^{2}$-scalar product on $\Omega^{0, \bullet}(X,
L^{p}\otimes E)$ is given
by
\begin{align}\label{1.2}
   \langle s_{1}, s_{2}\rangle=\int_{X}\big\langle s_{1}(x),
   s_{2}(x)\big\rangle dv_{X}(x).
   \end{align}
Let $\overline{\partial} ^{L^{p}\otimes E, \ast}$ be the formal adjoint of the
Dolbeault operator $\overline{\partial} ^{L^{p}\otimes E}$  with respect to the scalar
product (\ref{1.2}).
\begin{defn}
The Hodge-Dolbeault operator is defined by
\begin{align}\label{1.3}
 D_{p}=\sqrt{2}\big(\overline{\partial} ^{L^{p}\otimes E}+\overline{\partial} ^{L^{p}\otimes E,
 \ast}\big):\ \Omega^{0,\bullet}(X, L^{p}\otimes E)\rightarrow \Omega^{0,\bullet}(X, L^{p}\otimes E).
 \end{align}
 \end{defn}

\noindent
Then \begin{align}D^{2}_{p}=2\big(\overline{\partial} ^{L^{p}\otimes E}\overline{\partial} ^{L^{p}\otimes E,
 \ast}+\overline{\partial} ^{L^{p}\otimes E,
 \ast}\overline{\partial} ^{L^{p}\otimes E}\big),
 \end{align}
 which preserves the $\mathbb{Z}$-grading on $\Omega^{0,\bullet}(X, L^{p}\otimes E)$.

 By Hodge theory, we know that
 \begin{align}\label{1.4}
 \textrm{Ker} D^2_{p}|_{\Omega^{0,j}(X, L^{p}\otimes E)}\simeq H^{0,
 j}(X, L^p\otimes E),
 \end{align}
 where $H^{0, \bullet}(X, L^{p}\otimes E)$ denotes the Dolbeault cohomology group.
 By Andreotti-Grauert coarse vanishing theorem (see e.g. \cite[(1.29)]{Ma06}, \cite[(8.2.18)]{Ma07}) we obtain that
 for $p$ large enough,
 \begin{align}\label{1.19}
 H^{0,j}(X,L^{p}\otimes E)=0, \ \ \textrm{for} \ j\neq q.
 \end{align}
 It is a consequence of (\ref{1.4}) and (\ref{1.19}) that the kernel of $D^2_{p}$ is concentrated in
 degree $q$ for $p$ large enough. Let $P_{p}^{0,q}$ be the orthogonal projection from $\Omega^{0,q}(X,L^{p}\otimes E)$
 on $\Ker(D^2_p)$, and let $P_{p}^{0,q}(\cdot,\cdot)$ be its kernel with respect to $dv_X$. The
operator $P_{p}^{0,q}$ is smoothing, so the kernel $P_{p}^{0,q}(\cdot,\cdot)$ is smooth.

Let $I_{\det(\overline{W}^{\ast}) \otimes E}$ be the orthogonal
projection from $\Lambda (T^{\ast(0,1)}X)\otimes E$ onto
$\det(\overline{W}^{\ast}) \otimes E$. We denote by
$\big(\det(\overline{W}^{\ast})\big)^{\bot}$ the orthogonal
complement of $\det(\overline{W}^{\ast})$ in $\Lambda
(T^{\ast(0,1)}X)$. Denote by $\Theta$ the K\"ahler form associated to $g^{TX}$, i.e.,
\begin{align}\label{7.1}
\Theta(U, V)=\big\langle JU, V\big\rangle\ \  \textup{for}\ U, V\in TX.\end{align} The following diagonal asymptotic expansion of the kernel $P^{0,q}_{p}(\cdot,\cdot)$ was derived by
Ma and Marinescu, see \cite[Th.\,1.7]{Ma06}.

\begin{thm}\label{t1.1}
There exist smooth coefficients ${\bf b}_{r}(x)\in \textup{End}\big(\Lambda^{q}(T^{\ast(0,1)}X)\otimes E\big)_{x}$,
which are polynomials in $R^{TX}, R^{E} \ (\textup{and}\ d\Theta, R^{L})$ and
their derivatives of order $\leqslant 2r-2\ (\textup{resp}. \leqslant 2r-1, 2r)$ at $x$, such that
\begin{align}
{\bf b}_{0}=I_{\textup{det}(\overline{W}^{\ast})\otimes E}
\end{align}
and for any $k,l\in \mathbb{N}$, there exists $C_{k,l}>0$ with
\begin{align}\label{1.22}
\Big|P^{0,q}_{p}(x,x)-\sum^{k}_{r=0}{\bf b}_{r}(x)p^{n-r} \Big|_{C^{l}(X)}\leqslant C_{k,l}p^{n-k-1}
\end{align}
for any $p\in \mathbb{N}$. Moreover, the expansion is uniform in that for any
$k,l\in \mathbb{N}$, there is an integer $s$ such that if all data
$(g^{TX}, h^{L}, \nabla^{L},h^{E}, \nabla^{E})$ run over a set which are bounded in $C^{s}$ and with
$g^{TX}$ bounded below, there exists the constant $C_{k,l}$ independent of $g^{TX}$, and
the $C^{l}$-norm in $\textup{(\ref{1.22})}$ includes also the derivative on the parameters.
\end{thm}

The purpose of this paper is to calculate the second coefficient ${\bf b}_{1}$ in the asymptotic expansion (\ref{1.22}).
The readers are referred to the monograph \cite{Ma07} for a comprehensive study of the asymptotic expansion
of Bergman kernels and the methods of calculation of the coefficients along the lines of the present paper.

To state our main result we continue to introduce more notations.
Let $\nabla^{TX}$ denote the Levi-Civita connection on $(TX, g^{TX})$,
and let $\nabla^{T^{(1, 0)}X}$ be the Chern
connection of $(T^{(1, 0)}X$, $h^{T^{(1, 0)}X})$,
where $h^{T^{(1, 0)}X}$ is the Hermitian metric on $T^{(1, 0)}X$
induced by $g^{TX}$ in (\ref{1.12a}). We denote by $R^{T^{(1,0)}X}$ the curvature of $\nabla^{T^{(1,0)}X}$.
Let ${\bf J}: TX\rightarrow TX$ be the almost complex structure defined by
\begin{align}\label{1.12}
 \omega(U, V)=g^{TX}({\bf J}U, V) \ \ \textrm{for} \ U, V\in TX.
 \end{align}
Then $J$ commutes with ${\bf J}$.
Let $v_{1}, \ldots, v_{n}$ an orthonormal frame of $(T^{(1, 0)}X, h^{T^{(1, 0)}X})$
such that the subbundle $W$ is spanned by $v_{1}, \ldots, v_{q}$, and let
$v^{1}, \ldots, v^{n}$ be the dual frame.  It is a consequence of (\ref{1.12a}) and (\ref{1.12}) that
\begin{align}\label{1.13}
 {\bf J}v_{j}=-\sqrt{-1}v_{j},\ \textrm{for}\ j\leqslant q;\ \ \  {\bf J}v_{j}=\sqrt{-1}v_{j},\
 \textrm{for}\ j\geqslant q+1.
\end{align} Let $T_{\bf J}^{(1, 0)}X$ and $T_{\bf J}^{(0, 1)}X$ be the eigenbundles of ${\bf J}$ corresponding
to the eignevalues $\sqrt{-1}$ and $-\sqrt{-1}$ respectively. Set $u_{j}=\overline{v}_{j}$ if $j\leqslant q$ and
$u_{j}=v_{j}$ otherwise. Then $u_{1}, \ldots, u_{n}$ forms an orthonormal frame of the subbundle $T_{\bf J}^{(1, 0)}X$
and \begin{align}\label{3.2}
\omega=\sqrt{-1}\sum^{n}_{j=1}u^{j}\wedge \overline{u}^{j}.
\end{align}

Let $\nabla^{B}$ be the Bismut connection (see (\ref{1.12b})) on $\Lambda (T^{\ast(0, 1)}X)$
whose curvature is denoted by $R^{B}$. Denote by $\nabla^{X}\psi, \nabla^{B}\psi$ the covariant derivative of a tensor
$\psi$ with respect to $\nabla^{TX}$ and $\nabla^{B}$, respectively.
Let $e_{1}, \ldots, e_{2n}$ denote an orthonormal frame of $(TX, g^{TX})$, set
\begin{align}\label{1.12g}
\big|\nabla^{X}{\bf J}\big|^{2}=\sum_{i, j=1}^{2n}\big|(\nabla_{e_{i}}^{X}{\bf J})e_{j}\big|^{2},\ \
\big|\nabla^{B}{\bf J}\big|^{2}=\sum_{i, j=1}^{2n}\big|(\nabla_{e_{i}}^{B}{\bf J})e_{j}\big|^{2};
\end{align}
We denote by $T_{as}$ the anti-symmetrization of the torsion tensor of the connection induced by
the Chern connection $\nabla^{T^{(1,0)}X}$ on $TX$ (cf. (\ref{1.5}), (\ref{1.6})).
Let $\Lambda_{\omega}$ be the contraction operator with the form $\omega$.
Let $P$ be the smooth $2$-form over $X$ defined by
\begin{align}
P(U, V)=&\frac{1}{2}\big\langle R^{B}(u_{j}, \overline{u}_{j})U, V\big\rangle
\nonumber \\&+\frac{1}{4}(dT_{as})(u_{j}, \overline{u}_{j}, U, V)+
\Big(\frac{1}{2}\textup{Tr}\big[R^{T^{(1, 0)}X}\big]+R^{E}\Big)(U, V).
\end{align}
The summation convention of summing over repeated indices is used
here and throughout in this paper. Note that we have (cf. \cite[(1.2.51)]{Ma07}):
\begin{align}
T_{as}=-\sqrt{-1}\big(\partial-\overline{\partial}\big)\Theta,\ \
dT_{as}=2\sqrt{-1} \partial\overline{\partial}\Theta.
\end{align}

The main result in this paper is as follows.

\begin{thm}\label{t2.1} Let $X$ be a compact complex manifold and $(L,h^{L})$ be a holomorphic
Hermitian line bundle whose curvature is non-degenerate of signature
$(q,n-q)$. Let $(E,h^{E})$ be a holomorphic Hermitian vector bundle.
Endow $\Omega^{0,\bullet}(X, L^{p}\otimes E)$ with the $L^2$-scalar product induced by
the Riemannian metric $g^{TX}$ defined by $\textup{(\ref{1.12a})}$ and by $h^{L},h^{E}$. Then the
coefficient ${\bf b}_1$ from the expansion $\textup{(\ref{1.22})}$ of the Bergman kernel
$P^{0,q}_p(\cdot, \cdot)$ on $(0,q)$-forms is given by
\begin{align}\label{2.1}
\pi {\bf b}_{1}(x)=&
\bigg[\frac{1}{2} R^{E}(u_{j}, \overline{u}_{j})+\frac{1}{4}\textup{Tr}
\big[R^{T^{(1,0)}X}\big](u_{j}, \overline{u}_{j})
\nonumber\\&
-\frac{1}{16}\Lambda_{\omega}\big(d(\Lambda_{\omega}T_{as})\big)-\frac{1}{144}\big|(\nabla_{u_{i}}^{B}{\bf J})u_{j}\big|^{2}
\bigg]I_{\det({\overline W}^{\ast})\otimes E}
\\&
+\frac{1}{72}\sum^{q}_{i,j=1}\sum^{n}_{k,l=q+1}
\Big\langle(\nabla_{u_{m}}^{B}{\bf J})u_{j}, u_{k}\Big\rangle \Big\langle(\nabla_{\overline{u}_{m}}^{B}{\bf J})
\overline{u}_{i}, \overline{u}_{l}\Big\rangle \overline{u}^{l}\wedge i_{u_{i}}
I_{\det(\overline{W}^{\ast})\otimes E}u^{j}\wedge i_{\overline{u}_{k}}
\nonumber \\&-\frac{1}{4}
\sum_{j=1}^{q}\sum_{k=q+1}^{n}\Big[P(\overline{u}_{j}, \overline{u}_{k})
-\frac{\sqrt{-1}}{3}\big\langle(\nabla^{B}\nabla^{B}{\bf J})_{(u_{i}, \overline{u}_{i})}\overline{u}_{j}, \overline{u}_{k}\big\rangle \Big]\overline{u}^{k}\wedge
 i_{u_{j}}I_{\det(\overline{W}^{\ast})\otimes E}
\nonumber  \\&+
\frac{1}{8}\sum_{i,j=1}^{q}\sum_{k,l=q+1}^{n}\Big[\frac{1}{8}(dT_{as})(\overline{u}_{i}, \overline{u}_{j}, \overline{u}_{k}, \overline{u}_{l})-\frac{1}{15}\big\langle (\nabla_{u_{m}}^{B}{\bf J})\overline{u}_{i}, \overline{u}_{l}\big\rangle\cdot
\big\langle (\nabla_{\overline{u}_{m}}^{B}{\bf J})\overline{u}_{j}, \overline{u}_{k}\big\rangle
\nonumber \\& \ \ \ \ \ \ \ \ \ \ \ \ \ \ \ \ \ \ \
-\frac{1}{10}\big\langle (\nabla_{\overline{u}_{m}}^{B}{\bf J})\overline{u}_{i}, \overline{u}_{l}\big\rangle\cdot
\big\langle (\nabla_{u_{m}}^{B}{\bf J})\overline{u}_{j}, \overline{u}_{k}\big\rangle \Big]
\overline{u}^{k}\wedge \overline{u}^{l}
\wedge i_{u_{i}} i_{u_{j}}I_{\det(\overline{W}^{\ast})\otimes E}
\nonumber \\&-\frac{1}{4}
\sum_{j=1}^{q}\sum_{k=q+1}^{n}\Big[P(u_{k}, u_{j})
-\frac{\sqrt{-1}}{3}\big\langle(\nabla^{B}\nabla^{B}{\bf J})_{(\overline{u}_{i}, u_{i})}u_{k}, u_{j}\big\rangle\Big]I_{\det(\overline{W}^{\ast})\otimes E}
u^{j}\wedge i_{\overline{u}_{k}}
\nonumber \\&+
\frac{1}{8}\sum_{i,j=1}^{q}\sum_{k,l=q+1}^{n}
\Big[\frac{1}{8}(dT_{as})(u_{i}, u_{j}, u_{k}, u_{l})
-\frac{1}{15}\big\langle (\nabla_{\overline{u}_{m}}^{B}{\bf J})u_{i}, u_{l}\big\rangle\cdot
\big\langle (\nabla_{u_{m}}^{B}{\bf J})u_{j}, u_{k}\big\rangle
\nonumber \\& \ \ \ \ \ \ \ \ \ \ \ \ \ \ \ \ \ \ \
-\frac{1}{10}\big\langle (\nabla_{u_{m}}^{B}{\bf J})u_{i}, u_{l}\big\rangle\cdot
\big\langle (\nabla_{\overline{u}_{m}}^{B}{\bf J})u_{j}, u_{k}\big\rangle\Big]
I_{\det(\overline{W}^{\ast})\otimes E} u^{i}\wedge u^{j}
\wedge i_{\overline{u}_{l}} i_{\overline{u}_{k}}.\nonumber
\end{align}
In particular,
\begin{align}\begin{split}\label{2.1b}
& \pi\cdot \textup{Tr}|_{\Lambda^{q}(T^{\ast(0,1)}X)}\big[{\bf b}_{1}(x)\big]
\\=&\frac{1}{2}R^{E}(u_{j}, \overline{u}_{j})+\frac{1}{4}
\textup{Tr}\big[R^{T^{(1,0)}X}\big](u_{j}, \overline{u}_{j})
-\frac{1}{16}\Lambda_{\omega}\big(d(\Lambda_{\omega}T_{as})\big).
\end{split}\end{align}
\end{thm}

\noindent Note that  Hsiao \cite{Hsiao12} independently calculated by other methods the coefficient ${\bf b}_{1}$
for the trivial line bundle with mixed curvature over $\mathbb{C}^{n}$ endowed with the Euclidean metric.

By integrating $P_{p}^{0, q}(x, x)$ over $X$ we obtain $\dim H^{0, q}(X, L^{p}\otimes E)$
which by (\ref{1.19}) equals the Euler characteristic of $L^{p}\otimes E$ for large $p$. This is in turn given by
the Riemann-Roch-Hirzebruch formula (\ref{1.19a}). Thus, by integration of the asymptotic expansion (\ref{1.22})
we can compare coefficients with the Riemamm-Roch-Hirzebruch formula and we can check our formulas for ${\bf b}_{1}(x)$.
This will be carried out in $\S {\ref{s6}}$.

Since the explicit formula (\ref{2.1}) seems rather lengthy, it is worthwhile to show
what it reduces to in various interesting special cases.
Denote by $R^{TX}, r^{X}$ the curvature and scalar curvature of the Levi-Civita connection $\nabla^{TX}$, respectively.
\begin{cor}\label{t1.2}
If $(X, g^{TX}, J)$ is K\"ahler, then we have
\begin{align}\label{1.12e}
\pi {\bf b}_{1}(x)=&
\bigg[\frac{1}{2}R^{E}(u_{j}, \overline{u}_{j})+\frac{1}{4}\textup{Tr}
\big[R^{T^{(1,0)}X}\big](u_{j}, \overline{u}_{j})-\frac{1}{144}\big|(\nabla_{u_{i}}^{X}{\bf J})u_{j}\big|^{2}
\bigg]I_{\det({\overline W}^{\ast})\otimes E}
\nonumber \\&
+\frac{1}{72}\sum^{q}_{i,j=1}\sum^{n}_{k,l=q+1}
\Big\langle(\nabla_{u_{m}}^{X}{\bf J})u_{j}, u_{k}\Big\rangle \Big\langle(\nabla_{\overline{u}_{m}}^{X}{\bf J})
\overline{u}_{i}, \overline{u}_{l}\Big\rangle \overline{u}^{l}\wedge i_{u_{i}}
I_{\det(\overline{W}^{\ast})\otimes E}u^{j}\wedge i_{\overline{u}_{k}}
\nonumber \\&-\frac{1}{4}
\sum_{j=1}^{q}\sum_{k=q+1}^{n}\Big[\big(\frac{1}{2}\textup{Tr}\big[R^{T^{(1, 0)}X}\big]+R^{E}\big)(\overline{u}_{j}, \overline{u}_{k})
 \\&\ \ \ \ \ \ \ \ \ \ \ \ \ \ \ \ \ \ \
-\frac{1}{6}\big\langle R^{TX}(u_{i}, \overline{u}_{i})\overline{u}_{j}, \overline{u}_{k}\big\rangle\Big]
\overline{v}^{k}\wedge
 i_{\overline{v}_{j}}I_{\det(\overline{W}^{\ast})\otimes E}
\nonumber  \\&-\frac{1}{4}
\sum_{j=1}^{q}\sum_{k=q+1}^{n}\Big[\big(\frac{1}{2}\textup{Tr}\big[R^{T^{(1, 0)}X}\big]+R^{E}\big)(u_{k}, u_{j})
\nonumber \\&\ \ \ \ \ \ \ \ \ \ \ \ \ \ \ \ \ \ \
-\frac{1}{6}\big\langle R^{TX}(u_{i}, \overline{u}_{i})u_{k}, u_{j}\big\rangle\Big]
I_{\det(\overline{W}^{\ast})\otimes E}
\overline{v}^{j}\wedge i_{\overline{v}_{k}}. \nonumber
\end{align}
Taking the trace over $\Lambda^{q}(T^{\ast(0, 1)}X)$ yields
\begin{align}
\pi\cdot \textup{Tr}|_{\Lambda^{q}(T^{\ast(0,1)}X)}\big[{\bf b}_{1}(x)\big]
=\frac{1}{2}R^{E}(u_{j}, \overline{u}_{j})+\frac{1}{4}
\textup{Tr}\big[R^{T^{(1,0)}X}\big](u_{j}, \overline{u}_{j}).
\end{align}
\end{cor}

\begin{cor}\label{t1.3}
If $q=0$, then it follows $\textup{(\ref{1.12a})}$ and $\textup{(\ref{1.12})}$ that $(X, g^{TX}, J)$ is K\"ahler. Then
the formula $\textup{(\ref{2.1})}$ reduces to the known one \cite[(4.1.8)]{Ma07} for positive line bundles:
\begin{align}\label{1.12f}
\pi {\bf b}_{1}(x)=
\frac{1}{2}R^{E}(v_{j}, \overline{v}_{j})+\frac{1}{4}\textup{Tr}
\big[R^{T^{(1,0)}X}\big](v_{j}, \overline{v}_{j})
\textup{Id}_{E}=\frac{1}{2} R^{E}(v_{j}, \overline{v}_{j})+\frac{r^{X}}{8}\textup{Id}_{E}.
\end{align}
\end{cor}

\noindent Formula (\ref{1.12f}) follows immediately from (\ref{1.12e}).

\section{Comparison between curvatures of Bismut and Levi-Civita connections}\label{s2}
In this section we introduce the notion of Bismut connection and develop some properties of the tensors $\nabla^{X}{\bf J}$ and
$\nabla^{B}{\bf J}$. We also derive a formula for the difference between $R^{B}$ and $R^{TX}$.

We use freely the notions from the Introduction.
The Chern connection $\nabla^{T^{(1,0)}X}$ on $T^{(1,0)}X$ induces naturally
a Hermitian connection $\nabla^{T^{(0,1)}X}$ on
$T^{(0,1)}X$. Set
\begin{align}\label{1.5}
\tilde{\nabla}^{TX}=\nabla^{T^{(1,0)}X}\oplus \nabla^{T^{(0,1)}X}.
\end{align}
Then $\tilde{\nabla}^{TX}$ is a Hermitian connection on
$TX\otimes_{\mathbb{R}}\mathbb{C}$ which preserves the
decomposition $TX\otimes_{\mathbb{R}}\mathbb{C}=T^{(1,0)}X\oplus
T^{(0,1)}X$. We still denote by $\tilde{\nabla}^{TX}$ the induced
connection on $TX$. Let $T$ be the torsion of the connection
$\tilde{\nabla}^{TX}$, and let $T_{as}$ be the anti-symmetrization of the
tensor $T$, i.e., for $U, V, W\in TX$,
\begin{align}\label{1.6}
  T_{as}(U,V,W)=\big\langle T(U,V), W\big\rangle +\big\langle T(V,W), U\big\rangle
  +\big\langle T(W, U), V \big\rangle.
  \end{align}
By (\ref{1.5}), the torsion operator $T$ maps $T^{(1,0)}X\otimes T^{(1,0)}X$
(resp. $T^{(0,1)}X\otimes T^{(0,1)}X$) into $T^{(1,0)}X$ (resp. $T^{(0,1)}X$) and
vanishes on $T^{(1,0)}X\otimes T^{(0,1)}X$.

Denote by $S^{B}$ the $1$-form with values in the antisymmetric element of $\textup{End}(TX)$
which satisfies for $U, V, W\in TX$,
\begin{align}\label{1.11a}
\big\langle S^{B}(U)V, W\big\rangle=-\frac{1}{2}T_{as}(U, V, W).
\end{align}
Then the Bismut connection on $TX$ is defined by
\begin{align}\label{1.12b}
\nabla^{B}=\nabla^{TX}+S^{B}.
\end{align}
By \cite[Prop.\,2.5]{Bismut89}, $\nabla^{B}$ preserves the metric $g^{TX}$ and the complex structure $J$ of $TX$.
Then the curvature $R^{B}$ is compatible with the complex structure $J$ of $TX$.

We now continue with some elementary observations about the tensor $\nabla^{X}{\bf J}$.

It follows from (\ref{1.12}) that ${\bf J}, \nabla_{\bullet}^{X}{\bf J},
(\nabla^{X}\nabla^{X}{\bf J})_{(\bullet, \bullet)}$
are skew-adjoint endomorphisms of $TX$ and that for $U,V,W\in TX$, we have
\begin{align}\label{1.14}
\Big\langle\big(\nabla^{X}_{U}{\bf J}\big)V, W\Big\rangle=
\big(\nabla^{X}_{U}\omega\big)(V, W),
\end{align}
which implies immediately
\begin{align}\label{1.14a}
\Big\langle\big(\nabla^{X}_{U}{\bf J}\big)V, W\Big\rangle+
\Big\langle\big(\nabla^{X}_{V}{\bf J}\big)W, U\Big\rangle+&
\Big\langle\big(\nabla^{X}_{W}{\bf J}\big)U, V\Big\rangle=d\omega(U,V,W)=0.
\end{align}

By the definition of $(\nabla^{X}\nabla^{X}{\bf J})_{(U, V)}$,
\begin{align}\label{1.15}
(\nabla^{X}\nabla^{X}{\bf J})_{(U, V)}-
(\nabla^{X}\nabla^{X}{\bf J})_{(V, U)}=\big[R^{TX}(U,V),{\bf J}\big].
\end{align}
From ${\bf J}^{2}=-1$, we obtain
\begin{align}\label{1.14b}
(\nabla_{U}^{X}{\bf J}){\bf J}+{\bf J}(\nabla_{U}^{X}{\bf J})=0
\end{align}
and
\begin{align}\label{1.14c}
 {\bf J}\cdot(\nabla^{X}\nabla^{X}{\bf J})_{(U, V)}+(\nabla_{U}^{X}{\bf J})\cdot
(\nabla_{V}^{X}{\bf J}) +(\nabla_{V}^{X}{\bf J})\cdot
(\nabla_{U}^{X}{\bf J})+
(\nabla^{X}\nabla^{X}{\bf J})_{(U, V)}\cdot{\bf J}=0.
\end{align}

\noindent
From (\ref{1.14a}), we have for $Y\in TX$,
\begin{align}\label{1.16}
\Big\langle (\nabla^{X}\nabla^{X}{\bf J})_{(Y, U)}V, W\Big\rangle
+\Big\langle (\nabla^{X}\nabla^{X}{\bf J})_{(Y, V)}W, U\Big\rangle
+\Big\langle (\nabla^{X}\nabla^{X}{\bf J})_{(Y, W)}U, V\Big\rangle=0.
\end{align}

\noindent
Let $T_{\bf J}^{\ast(1,0)}X$ and $T_{\bf J}^{\ast(0,1)}X$
be the dual bundle of $T_{\bf J}^{(1,0)}X$ and $T_{\bf J}^{(0,1)}X$, respectively.
From (\ref{1.14a}) and (\ref{1.14b}), we find that
\begin{align}\label{2.14}
\big\langle(\nabla_{\bullet}^{X}{\bf J})\ {\sbt}\ , \ {\sbt}\ \big\rangle
\textup{ is of type} \ \big(T_{\bf J}^{\ast(1,0)}X\big)^{\otimes 3}\oplus
\big(T_{\bf J}^{\ast(0,1)}X\big)^{\otimes 3}.\end{align}

On the other hand for the tensor $\nabla^{B}{\bf J}$,  we have for $U, V, W\in TX$,
\begin{align}
\Big\langle\big(\nabla^{B}_{U}{\bf J}\big)V, W\Big\rangle=
\big(\nabla^{B}_{U}\omega\big)(V, W)
\end{align}
and
\begin{align}\label{1.15a}
(\nabla^{B}\nabla^{B}{\bf J})_{(U, V)}-
(\nabla^{B}\nabla^{B}{\bf J})_{(V, U)}=\big[R^{B}(U,V),&{\bf J}\big].
\end{align}
From ${\bf J}^{2}=-1$ we obtain
\begin{align}
 {\bf J}\cdot(\nabla^{B}\nabla^{B}{\bf J})_{(U, V)}+(\nabla_{U}^{B}{\bf J})\cdot
(\nabla_{V}^{B}{\bf J}) +(\nabla_{V}^{B}{\bf J})\cdot
(\nabla_{U}^{B}{\bf J})+
(\nabla^{B}\nabla^{B}{\bf J})_{(U, V)}\cdot{\bf J}=0.
\end{align}

Since $\nabla^{B}$ is not torsion-free, then the analogue of  (\ref{1.14a})
does not hold for the tensor $\nabla^{B}{\bf J}$, neither does the analogue of (\ref{1.16}).
Hence $\nabla^{B}{\bf J}$ does not admit such a decomposition as (\ref{2.14}).
In spite of this, $\nabla^{B}{\bf J}$ does satisfy the following property.

\begin{prop}\label{t3.1}
$\nabla^{B}{\bf J}$ preserves $T^{(1, 0)}X$ and $T^{(0, 1)}X$. Furthermore, it exchanges the subbundles
$W$ and $W^{\bot}$.
\end{prop}
\begin{proof} It is a consequence of the facts that $\nabla^{B}J=0$ (cf.\,\cite[Prop.\,2.5]{Bismut89})
and $J{\bf J}={\bf J}J$ that
\begin{align}
J(\nabla^{B}{\bf J})=(\nabla^{B}{\bf J})J,
\end{align}
which implies immediately that $\nabla^{B}{\bf J}$ preserves $T^{(1, 0)}X$ and $T^{(0, 1)}X$. Note that
the metric $\big\langle\cdot, \cdot \big\rangle$ is $\mathbb{C}$-bilinear. Then
for $1\leqslant j, k\leqslant q$ and $U\in TX$, we have
\begin{align}
\Big\langle(\nabla_{U}^{B}{\bf J})v_{j}, \overline{v}_{k}\Big\rangle=
\Big\langle\nabla_{U}^{B}({\bf J}v_{j}), \overline{v}_{k}\Big\rangle+
\Big\langle\nabla_{U}^{B}v_{j}, {\bf J}\overline{v}_{k}\Big\rangle=0.
\end{align}
Similarly, for  $q+1\leqslant j, k\leqslant n$,
\begin{align}
\Big\langle(\nabla_{U}^{B}{\bf J})v_{j}, \overline{v}_{k}\Big\rangle=0.
\end{align} This completes the proof of Proposition \ref{t3.1}.
\end{proof}

\begin{lemma}\label{t3.3}
\begin{align}\begin{split}\label{3.8}
\sum^{q}_{j=1}\sum_{k=q+1}^{n}\Big|\big\langle (\nabla_{\overline{u}_{i}}^{B}{\bf J})u_{j}, u_{k}\big\rangle \Big|^{2}
&=2\Big|\big\langle S^{B}(\overline{u}_{i})u_{j}, u_{k}\big\rangle\Big|^{2},
\\ \sum^{q}_{j=1}\sum_{k=q+1}^{n}\Big|\big\langle (\nabla_{u_{i}}^{B}{\bf J})u_{j}, u_{k}\big\rangle \Big|^{2}
&=\frac{1}{4}\big|\nabla^{B}{\bf J}\big|^{2}-2\Big|\big\langle S^{B}(\overline{u}_{i})u_{j}, u_{k}\big\rangle\Big|^{2}.
\end{split}\end{align}
\end{lemma}
\begin{proof}
In view of (\ref{2.14}) and Proposition \ref{t3.1}, we find
\begin{align}\label{3.9}
\sum^{q}_{j=1}\sum_{k=q+1}^{n}\Big|\big\langle (\nabla_{\overline{u}_{i}}^{B}{\bf J})u_{j}, u_{k}\big\rangle \Big|^{2}
=\frac{1}{2}\Big|\big\langle (\nabla_{\overline{u}_{i}}^{B}{\bf J})u_{j}, u_{k}\big\rangle \Big|^{2}
=2\Big|\big\langle S^{B}(\overline{u}_{i})u_{j}, u_{k}\big\rangle\Big|^{2}.
\end{align}
Again by Proposition \ref{t3.1} we get
\begin{align}\begin{split}\label{3.10}
\big|\nabla^{B}{\bf J}\big|^{2}&=2\big|(\nabla_{\overline{u}_{i}}^{B}{\bf J})u_{j}\big|^{2}
+2\big|(\nabla_{u_{i}}^{B}{\bf J})u_{j}\big|^{2}
\\ &=
4\sum^{q}_{j=1}\sum_{k=q+1}^{n}\Big|\big\langle (\nabla_{\overline{u}_{i}}^{B}{\bf J})u_{j}, u_{k}\big\rangle \Big|^{2}
+4\sum^{q}_{j=1}\sum_{k=q+1}^{n}\Big|\big\langle (\nabla_{u_{i}}^{B}{\bf J})u_{j}, u_{k}\big\rangle \Big|^{2}.
\end{split}\end{align}
Combining (\ref{3.9}) and (\ref{3.10}), we obtain the second equality of (\ref{3.8}). This completes the proof of Lemma \ref{t3.3}.
\end{proof}

The main result of this section is the following formula for the difference $R^{B}-R^{TX}$.
\begin{prop}\label{t2.4c} We have
\begin{align}\begin{split}\label{3.15}
&\Big\langle \big(R^{B}-R^{TX}\big)(u_{i}, \overline{u}_{i})u_{j}, \overline{u}_{j}\Big\rangle
\\=&
\frac{1}{8}\Big(\big|\nabla^{B}{\bf J}\big|^{2}-\big|\nabla^{X}{\bf J}\big|^{2}\Big)+\frac{1}{4}\Lambda_{\omega}\big(d(\Lambda_{\omega}T_{as})\big)
-2\Big|\big\langle S^{B}(\overline{u}_{i})u_{j}, u_{k}\big\rangle \Big|^{2}.
\end{split}\end{align}
\end{prop}

\noindent The proof of Proposition \ref{t2.4c} is based on the following three Lemmas.

\begin{lemma}\label{t2.4a}
\begin{align}\begin{split}\label{3.3}
&\Big\langle \big(R^{B}-R^{TX}\big)(u_{i}, \overline{u}_{i})u_{j}, \overline{u}_{j}\Big\rangle
\\=&
\Big|\big\langle S^{B}(u_{i})u_{j}, u_{k}\big\rangle\Big|^{2}-\Big|\big\langle S^{B}(\overline{u}_{i})u_{j}, u_{k}\big\rangle\Big|^{2}
+\frac{1}{16}\Lambda_{\omega}\Lambda_{\omega}(dT_{as}).
\end{split}\end{align}
\end{lemma}

\begin{proof}
One verifies directly that
\begin{align}\begin{split}\label{3.4}
&\Big\langle \big(R^{B}-R^{TX}\big)(u_{i}, \overline{u}_{i})u_{j}, \overline{u}_{j}\Big\rangle
\\=&
\Big\langle S^{B}(u_{i})u_{j}, S^{B}(\overline{u}_{i})\overline{u}_{j}\Big\rangle
-\Big\langle S^{B}(\overline{u}_{i})u_{j}, S^{B}(u_{i})\overline{u}_{j}\Big\rangle
\\&+\Big\langle \big(\nabla_{u_{i}}^{X}S^{B}\big)(\overline{u}_{i})u_{j}, \overline{u}_{j}\Big\rangle
- \Big\langle \big(\nabla_{\overline{u}_{i}}^{X}S^{B}\big)(u_{i})u_{j}, \overline{u}_{j}\Big\rangle
\end{split}\end{align}
By (\ref{1.11a}), we obtain
\begin{align}\begin{split}\label{3.5}
\Big\langle \big(\nabla_{u_{i}}^{X}S^{B}\big)(\overline{u}_{i})u_{j}, \overline{u}_{j}\Big\rangle
&=-\frac{1}{2}\big(\nabla_{u_{i}}^{X}T_{as}\big)(\overline{u}_{i}, u_{j}, \overline{u}_{j}),
\\ \Big\langle \big(\nabla_{\overline{u}_{i}}^{X}S^{B}\big)(u_{i})u_{j}, \overline{u}_{j}\Big\rangle
&=-\frac{1}{2}\big(\nabla_{\overline{u}_{i}}^{X}T_{as}\big)(u_{i}, u_{j}, \overline{u}_{j}).
\end{split}\end{align}
Substituting (\ref{3.5}) into (\ref{3.4}) yields (\ref{3.3}).
\end{proof}

\begin{lemma}\label{t3.2}
\begin{align}\begin{split}\label{3.1}
&\frac{1}{8}\Big(\big|\nabla^{B}{\bf J}\big|^{2}-\big|\nabla^{X}{\bf J}\big|^{2}\Big)
\\=&
\Big|\big\langle S^{B}(u_{i})u_{j}, u_{k}\big\rangle\Big|^{2}+\Big|\big\langle S^{B}(\overline{u}_{i})u_{j}, u_{k}\big\rangle\Big|^{2}
\\&
+\frac{\sqrt{-1}}{2}\Big\langle S^{B}(u_{i})u_{j}, (\nabla_{\overline{u}_{i}}^{X}{\bf J})\overline{u}_{j}\Big\rangle
-\frac{\sqrt{-1}}{2}\Big\langle S^{B}(\overline{u}_{i})\overline{u}_{j}, (\nabla_{u_{i}}^{X}{\bf J})u_{j}\Big\rangle.
\end{split}\end{align}
\end{lemma}
\begin{proof}
By (\ref{1.12g}) and (\ref{2.14}) we obtain
\begin{align}\label{3.6}
\big|\nabla^{X}{\bf J}\big|^{2}=2\big|(\nabla^{X}_{u_{i}}{\bf J})u_{j}\big|^{2}, \ \
\big|\nabla^{B}{\bf J}\big|^{2}=2\big|(\nabla^{B}_{u_{i}}{\bf J})u_{j}\big|^{2}+
2\big|(\nabla^{B}_{\overline{u}_{i}}{\bf J})u_{j}\big|^{2}.
\end{align}
Combining (\ref{1.12b}) and (\ref{3.6}) yields
\begin{align}\begin{split}\label{3.7}
&\big|\nabla^{B}{\bf J}\big|^{2}-\big|\nabla^{X}{\bf J}\big|^{2}
\\ =&
2\Big|[S^{B}(u_{i}), {\bf J}]u_{j}\Big|^{2}+2\Big|[S^{B}(\overline{u}_{i}), {\bf J}]u_{j}\Big|^{2}
\\&
+2\Big\langle [S^{B}(u_{i}), {\bf J}]u_{j}, \big(\nabla_{\overline{u}_{i}}^{X}{\bf J}\big)\overline{u}_{j}\Big\rangle
+2\Big\langle [S^{B}(\overline{u}_{i}), {\bf J}]\overline{u}_{j}, \big(\nabla_{u_{i}}^{X}{\bf J}\big)u_{j}\Big\rangle.
\end{split}\end{align}
Clearly,
\begin{align}\label{3.7a}
\Big|[S^{B}(u_{i}), {\bf J}]u_{j}\Big|^{2}=\Big|\big\langle [S^{B}(u_{i}), {\bf J}]u_{j}, u_{k}\big\rangle \Big|^{2}
=4\Big|\big\langle S^{B}(u_{i})u_{j}, u_{k}\big\rangle \Big|^{2},
\end{align}
and
\begin{align}\label{3.7b}
\Big|[S^{B}(\overline{u}_{i}), {\bf J}]u_{j}\Big|^{2}=\Big|\big\langle [S^{B}(\overline{u}_{i}), {\bf J}]u_{j}, u_{k}\big\rangle \Big|^{2}
=4\Big|\big\langle S^{B}(\overline{u}_{i})u_{j}, u_{k}\big\rangle \Big|^{2}.
\end{align}
Substituting (\ref{3.7a}) and (\ref{3.7b}) into (\ref{3.7}) yields (\ref{3.1}).
\end{proof}

\begin{lemma}\label{t3.4}
\begin{align}\begin{split}\label{3.11}
\frac{1}{4}\Lambda_{\omega}\big(d(\Lambda_{\omega}T_{as})\big)
=&\frac{1}{16}\Lambda_{\omega}\Lambda_{\omega}(dT_{as})
-\big\langle S^{B}(u_{i})u_{j}, u_{k}\big\rangle \big\langle \nabla_{\overline{u}_{i}}^{TX}\overline{u}_{j},
\overline{u}_{k}\big\rangle
\\& -\big\langle S^{B}(\overline{u}_{i})\overline{u}_{j}, \overline{u}_{k}\big\rangle
\big\langle \nabla_{u_{i}}^{TX}u_{j}, u_{k}\big\rangle.
\end{split}\end{align}
\end{lemma}

\begin{proof}
Clearly,
\begin{align}\begin{split}\label{3.12}
&-\frac{1}{4}\Lambda_{\omega}\big(d(\Lambda_{\omega}T_{as})\big)=\frac{\sqrt{-1}}{4}d(\Lambda_{\omega}T_{as})(u_{i}, \overline{u}_{i})
\\=&
\frac{1}{2}\nabla_{u_{i}}\big(T_{as}(u_{j}, \overline{u}_{j}, \overline{u}_{i})\big)
-\frac{1}{2}\nabla_{\overline{u}_{i}}\big(T_{as}(u_{j}, \overline{u}_{j}, u_{i})\big)
-\frac{1}{2}T_{as}(u_{j}, \overline{u}_{j}, [u_{i}, \overline{u}_{i}])
\end{split}\end{align}
By (\ref{1.11a}) we obtain
\begin{align}\begin{split}\label{3.13}
-\frac{1}{4}\Lambda_{\omega}\big(d(\Lambda_{\omega}T_{as})\big)=&-\frac{1}{16}\Lambda_{\omega}\Lambda_{\omega}(dT_{as})
+\big\langle S^{B}(\overline{u}_{i})\overline{u}_{j}, \nabla_{u_{i}}^{TX}u_{j}\big\rangle
+\big\langle S^{B}(u_{j})\overline{u}_{i}, \nabla_{u_{i}}^{TX}\overline{u}_{j}\big\rangle
\\&-\big\langle S^{B}(u_{i})\overline{u}_{j}, \nabla_{\overline{u}_{i}}^{TX}u_{j}\big\rangle
+\big\langle S^{B}(u_{i})u_{j}, \nabla_{\overline{u}_{i}}^{TX}\overline{u}_{j}\big\rangle.
\end{split}\end{align}
Clearly,
\begin{align}\begin{split}\label{3.14}
\big\langle S^{B}(\overline{u}_{i})\overline{u}_{j}, \nabla_{u_{i}}^{TX}u_{j}\big\rangle
+\big\langle S^{B}(u_{j})\overline{u}_{i}, \nabla_{u_{i}}^{TX}\overline{u}_{j}\big\rangle
=&
\big\langle S^{B}(\overline{u}_{i})\overline{u}_{j}, \overline{u}_{k}\big\rangle\cdot
\big\langle\nabla_{u_{i}}^{TX}u_{j}, u_{k}\big\rangle,
\\
\big\langle S^{B}(u_{i})u_{j}, \nabla_{\overline{u}_{i}}^{TX}\overline{u}_{j}\big\rangle
-\big\langle S^{B}(u_{i})\overline{u}_{j}, \nabla_{\overline{u}_{i}}^{TX}u_{j}\big\rangle
=&
\big\langle S^{B}(u_{i})u_{j}, u_{k}\big\rangle\cdot
\big\langle\nabla_{\overline{u}_{i}}^{TX}\overline{u}_{j}, \overline{u}_{k}\big\rangle.
\end{split}\end{align}
Substituting (\ref{3.14}) into (\ref{3.13}) yields (\ref{3.11}). The proof of Lemma \ref{t3.4}
is complete.
\end{proof}

\begin{proof}[Proof of Proposition \ref{t2.4c}] Formula (\ref{3.15}) follows immediately from
(\ref{3.3}), (\ref{3.1}) and (\ref{3.11}).
\end{proof}

\section{Bergman Kernel of the Hodge-Dolbeault Operator}\label{s1}

In this section we introduce the corresponding Lichnerowicz
formula for the square of the Hodge-Dolbeault operator
and calculate out the curvature operator of the connection $\nabla^{B, \Lambda^{0,\bullet}\otimes L^{p}\otimes E}$
which arises in the Lichnerowicz formula.

\subsection{Lichnerowicz formula for the Hodge-Dolbeault operator}\label{s1a}
For any $v\in TX\otimes_{\mathbb{R}}\mathbb{C}$ with
decomposition $v=v_{1,0}+v_{0,1}\in T^{(1,0)}X\oplus T^{(0,1)}X$,
let $\overline{v}^{\ast}_{1,0}$ be the metric dual of $v_{1,0}$.
Then $c(v)=\sqrt{2}(\overline{v}^{\ast}_{1,0}\wedge-i_{v_{0,1}})$
defines the Clifford action of $v$ on $\Lambda (T^{\ast(0,1)}X)$,
where $\wedge$ and $i$ denote the standard exterior and interior multiplication,
respectively.

The Chern connection $\nabla^{T^{(1,0)}X}$ on $T^{(1, 0)}X$ induces
the Chern connection $\nabla^{\det(T^{(1,0)}X)}$
on the line bundle $\det(T^{(1,0)}X)$.
Let $\nabla^{Cl}$ denote the Clifford connection on $\Lambda(T^{\ast(0,
1)}X)$ induced canonically by $\nabla^{TX}$ and $\nabla^{\det(T^{(1,0)}X)}$.

If $e^{1}, \ldots, e^{2n}$ denotes an orthonormal frame of $T^{\ast}X$,
then define
\begin{align}\label{1.7}
^{c}(e^{i_{1}}\wedge \cdots \wedge e^{i_{j}})=c(e_{i_{1}})\cdots c(e_{i_{j}}),
\ \ \textup{for}\ i_{1} < \cdots< i_{j}.
\end{align}
In this sense $^{c}B$ is defined for any $B\in \Lambda (T^{\ast}X)\otimes_{\mathbb{R}}\mathbb{C}$
by extending $\mathbb{C}$-linearity.

Recall that $T_{as}$ is defined by (\ref{1.6}). Take $U\in TX$. Let
\begin{align}\label{1.8}
\nabla_{U}^{B,\Lambda^{0, \bullet}}
=\nabla_{U}^{Cl}-\frac{1}{4} \ ^{c}(i_{U}T_{as}) \end{align}
denote the Hermitian connection on $\Lambda (T^{\ast(0, 1)}X)$ induced by $\nabla^{Cl}$ and $T_{as}$.
Then $\nabla^{B, \Lambda^{0, \bullet}}$ is the Clifford connection on the spinor bundle $\Lambda (T^{\ast(0, 1)}X)$
induced by $\nabla^{B}$ on $TX$ and $\nabla^{\det(T^{(1,0)}X)}$ on $\det(T^{(1,0)}X)$.
Here $\Lambda (T^{\ast(0, 1)}X)$ is formally $S(TX)\otimes \det(T^{(1,0)}X)^{1/2}$ and
$S(TX)$ is the fundamental spinor bundle for the (possibly nonexistent) spin structure on $TX$ and
$\det(T^{(1,0)}X)^{1/2}$ is the (possibly nonexistent) square root of $\det(T^{(1,0)}X)$ (cf. \cite[Appendix D,\,p 397]{Mi}).
The connection $\nabla^{B, \Lambda^{0, \bullet}}$ preserves the $\mathbb{Z}$-grading of $\Lambda(T^{\ast(0, 1)}X)$
(cf. \cite[(1.4.27)]{Ma07}).
If $v_{1}, \ldots, v_{n}$ denotes an orthonormal frame of $T^{(1, 0)}X$, set
\begin{align}
e_{2j-1}=\frac{1}{\sqrt{2}}(v_{j}+\overline{v}_{j}), \ \ e_{2j}=\frac{\sqrt{-1}}{\sqrt{2}}(v_{j}-\overline{v}_{j}).
\end{align}
Then $e_{1}, \ldots, e_{2n}$ forms an orthonormal frame of $TX$. Set
\begin{align}
\nabla^{TX}e_{j}=\Gamma^{TX}e_{j}, \ \
\nabla^{\textup{det}(T^{(1, 0)}X)}(v_{1}\wedge\cdots \wedge v_{n})
=\Gamma^{\textup{det}(T^{(1, 0)}X)}(v_{1}\wedge\cdots \wedge v_{n}).
\end{align}
Denote by $\overline{v}^{j}$ the metric dual of $v_{j}$.
It is a consequence of \cite[(1.3.5)]{Ma07} that $\nabla^{B, \Lambda^{0, \bullet}}$ is given,
with respect to the frame
$\{\overline{v}^{j_{1}}\wedge\cdots \wedge \overline{v}^{j_{k}}, \ 1\leqslant j_{1}<\cdots<j_{k}\leqslant n \}$
of $\Lambda (T^{\ast(0, 1)}X)$,  by the local formula
\begin{align}
d+\frac{1}{4}\big\langle \Gamma^{TX}e_{i}, e_{j}\big\rangle c(e_{i})c(e_{j})
+\frac{1}{2}\Gamma^{\textup{det}(T^{(1, 0)}X)}-\frac{1}{4}\ ^{c}(i_{\cdot}T_{as}).
\end{align}
Let $\Gamma^{B, \Lambda^{0, \bullet}}$ be the connection $1$-form of $\nabla^{B, \Lambda^{0, \bullet}}$
(associated to the above frame of $\Lambda (T^{\ast(0, 1)}X)$), i.e.,
\begin{align}\label{1.8a}
\Gamma^{B, \Lambda^{0, \bullet}}=\frac{1}{4}\big\langle \Gamma^{TX}e_{i}, e_{j}\big\rangle  c(e_{i})c(e_{j})
+\frac{1}{2}\Gamma^{\textup{det}(T^{(1, 0)}X)}-\frac{1}{4}\ ^{c}(i_{\cdot}T_{as}).
\end{align}

Denote by $\nabla^{L^{p}\otimes E}$ the Chern connection on $L^{p}\otimes E$ induced by
$\nabla^{L}$ and $\nabla^{E}$. Set
\begin{align}
\nabla^{B, \Lambda^{0,\bullet}\otimes L^{p}\otimes E}=\nabla^{B, \Lambda^{0, \bullet}}\otimes 1+1\otimes \nabla^{L^{p}\otimes E}.
\end{align}
Then $\nabla^{B, \Lambda^{0,\bullet}\otimes L^{p}\otimes E}$ is a Hermitian connection on $\Lambda (T^{\ast(0, 1)}X)\otimes L^{p}\otimes E$.
Let $R^{B, \Lambda^{0,\bullet}\otimes L^{p}\otimes E}$ be the curvature operator of $\nabla^{B, \Lambda^{0,\bullet}\otimes L^{p}\otimes E}$, and let $\Delta^{B, \Lambda^{0,\bullet}\otimes L^{p}\otimes E}$ be the Bochner Laplacian associated to $\nabla^{B, \Lambda^{0,\bullet}\otimes L^{p}\otimes E}$, i.e.,
 \begin{align}\label{1.9}
\Delta^{B, \Lambda^{0,\bullet}\otimes L^{p}\otimes E}=-\sum^{2n}_{j=1}\Big[\big(\nabla_{e_{j}}^{B, \Lambda^{0,\bullet}\otimes L^{p}\otimes E}\big)^2-
 \nabla^{B, \Lambda^{0,\bullet}\otimes L^{p}\otimes E}_{\nabla_{e_{j}}^{TX}e_{j}}\Big].
 \end{align}

If $e_{1}, \ldots, e_{2n}$ denotes an
orthonormal frame of $TX$, then set
\begin{align}\label{1.10}|A|^{2}=\sum_{i<j<k}|A(e_{i}, e_{j},
e_{k})|^2, \ \ \textup{for}\ A\in \Lambda^{3}(T^{\ast}X).
\end{align}
The following Lichnerowicz formula \cite[(1.4.29)]{Ma07} for
$D^2_{p}$ holds:
\begin{align}\label{1.11} \begin{split}
 D^{2}_{p}=&\Delta^{B, \Lambda^{0,\bullet}\otimes L^{p}\otimes E}+\frac{r^{X}}{4}+\frac{1}{2}pR^{L}(e_{i}, e_{j})c(e_{i})c(e_{j})
 \\&+  \ ^c\big(R^{E}+\frac{1}{2}\textrm{Tr}\big[R^{T^{(1,0)}X}\big]\big)
 -\frac{1}{4} \ ^c(d T_{as})-\frac{1}{8}\big|T_{as} \big|^2.
 \end{split}
  \end{align}
  If $A$ is a $2$-form, then
\begin{align}\begin{split}\label{1.16a}
\frac{1}{4} A( e_{i}, e_{j}) c(e_{i})c(e_{j})=&
-\frac{1}{2} A( v_{j},\overline{v}_{j}\big)
+ A\big( v_{j}, \overline{v}_{k}) \overline{v}^{k}\wedge i_{\overline{v}_{j}}
\\&+\frac{1}{2} A( v_{j}, v_{k})  i_{\overline{v}_{j}}i_{\overline{v}_{k}}
+\frac{1}{2} A ( \overline{v}_{j}, \overline{v}_{k})
\overline{v}^{j}\wedge \overline{v}^{k}\wedge.\end{split}
\end{align}
If $A$ is a skew-adjoint endomorphism of $TX$, then
\begin{align}\begin{split}\label{1.16b}
\frac{1}{4}\big\langle Ae_{i}, e_{j}\big\rangle c(e_{i})c(e_{j})=&
-\frac{1}{2}\big\langle Av_{j},\overline{v}_{j}\big\rangle
+\big\langle Av_{j}, \overline{v}_{k}\big\rangle \overline{v}^{k}\wedge i_{\overline{v}_{j}}
\\&+\frac{1}{2}\big\langle Av_{j}, v_{k}\big\rangle  i_{\overline{v}_{j}}i_{\overline{v}_{k}}
+\frac{1}{2}\big\langle A \overline{v}_{j}, \overline{v}_{k}\big\rangle
\overline{v}^{j}\wedge \overline{v}^{k}\wedge.\end{split}
\end{align}

Set
 \begin{align}\label{1.17}\begin{split}
 \omega_{d,x}=&-2\pi\sum^{q}_{j=1}i_{\overline{v}_{j}}\wedge
 \overline{v}^{j}-2\pi\sum^{n}_{j=q+1}\overline{v}^{j}\wedge
 i_{\overline{v}_{j}},
  \\
 \tau_x=&\pi \textrm{Tr}\big|_{TX}(-{\bf J}^2)^{1/2}=2n\pi. \end{split}
 \end{align}
From (\ref{1.12}), (\ref{1.13}), (\ref{1.16a}) and (\ref{1.17}), we have
 \begin{align}\label{1.18}
 \frac{1}{2}R^{L}(e_{i}, e_{j})c(e_{i})c(e_{j})=-2\omega_{d}-\tau.
 \end{align}

  \begin{defn}
  The Bergman kernel $P_{p}(x,y)\ (x, y\in X)$ is the smooth kernel
  with respect to $dv_{X}(y)$ of the orthogonal projection $P_{p}$
  from $\Omega^{0, \bullet}(X,L^{p}\otimes E)$ onto $\Ker(D_{p})$. \end{defn}
\noindent  Then $P_{p}(x,x)$ is an element of
  $\textrm{End}\big(\Lambda(T^{\ast(0,1)}X)\otimes E\big)_{x}$. It is a consequence of the Andreotti-Grauert
  vanishing theorem (\ref{1.19}) that the kernel $P^{0,q}_{p}(x,y)$
  coincides with the Bergman kernel $P_{p}(x,y)$ for $p$ large enough. Hence, Theorem \ref{t1.1} in fact gives
  the diagonal asymptotic expansion of the Bergman kernel $P_{p}(x,x)$.

\subsection{A formula of the curvature operator $R^{B, \Lambda^{0,\bullet}\otimes L^{p}\otimes E}$}
\label{s1b}
Set
\begin{align}\label{1.11d}
Q=[\nabla^{TX}, S^{B}]+S^{B}\wedge S^{B}.
\end{align}
Combining (\ref{1.12b}) and (\ref{1.11d}) yields
\begin{align}\label{1.12d}
R^{B}=R^{TX}+Q.
\end{align}

\noindent
By \cite[(3.4)]{Berline04}, (\ref{1.12d}) and the explanation after (\ref{1.8}), we know that
\begin{align}\label{7.3}
R^{B, \Lambda^{0, \bullet}}=
\frac{1}{4}\big\langle R^{B} e_{i}, e_{j}\big\rangle c(e_{i})c(e_{j})
+\frac{1}{2}\textup{Tr}\big[R^{T^{(1, 0)}X}\big].
\end{align}

\noindent
Of course, we can get (\ref{7.3}) from (\ref{1.8a}) by a direct computation.
Note that since $R^{B}$ is compatible with the complex structure $J$, we obtain from (\ref{1.16b}) and (\ref{7.3})
that $R^{B, \Lambda^{0, \bullet}}$ preserves the $\mathbb{Z}$-grading of $\Lambda(T^{\ast(0, 1)}X)$.
From (\ref{7.3}), we have
\begin{align}\label{1.11b}
R^{B, \Lambda^{0,\bullet}\otimes L^{p}\otimes E}=&R^{B, \Lambda^{0, \bullet}}+pR^{L}+R^{E}
\\ \nonumber =&
\frac{1}{4}\big\langle R^{B} e_{i}, e_{j}\big\rangle c(e_{i})c(e_{j})
+\frac{1}{2}\textup{Tr}\big[R^{T^{(1, 0)}X}\big]+pR^{L}+R^{E}.
\end{align}

\section{An Explicit formula of ${\bf b}_{1}$}
In this section we provide an explicit formula for the second coefficient ${\bf b}_{1}$
in the diagonal asymptotic expansion (\ref{1.22}).

\subsection{Trivialization}\label{s4a}
 Fix a point $x_{0}\in X$.
For $\varepsilon>0$, denote by $B^{X}(x_{0}, \varepsilon)$ (resp. $B^{T_{x_0}X}(0, \varepsilon)$) the open
ball in $X$ (resp. $T_{x_{0}}X$) with the center $x_{0}$ and radius $\varepsilon$.
Then we identify $B^{T_{x_0}X}(0, \varepsilon)$ with $B^{X}(x_{0}, \varepsilon)$
by the exponential map $Z\mapsto \textup{exp}^{X}_{x_{0}}(Z)$ for $Z\in T_{x_{0}}X$.
Let $v_{1}, \ldots, v_{n}$ be an orhonormal basis of $T^{(1, 0)}_{x_{0}}X$ such that
the equations (\ref{1.13}) hold at the point $x=x_{0}$.
Set
\begin{align}
e_{2j-1}=\frac{1}{\sqrt{2}}(v_{j}+\overline{v}_{j}), \
e_{2j}=\frac{\sqrt{-1}}{\sqrt{2}}(v_{j}-\overline{v}_{j}).
\end{align}
Then $e_{1}, \ldots, e_{2n}$ forms an orthonormal basis of $T_{x_{0}}X$. The coordinates on
$T_{x_{0}}X\simeq \mathbb{R}^{2n}$ is given by
\begin{align}
(Z_{1}, \cdots, Z_{2n})\in\mathbb{R}^{2n}\mapsto \sum_{j=1}^{2n}Z_{j}e_{j}\in T_{x_{0}}X.
\end{align}

We identify $L_{Z}, E_{Z}$ and $(E_{p})_{Z}$ for $Z\in B^{T_{x_{0}}X}(0,\varepsilon)$ to $L_{x_{0}}, E_{x_{0}}$
and $(E_{p})_{x_{0}}$ by parallel transport with respect to the connection $\nabla^{L},
\nabla^{E}$ and $\nabla^{B, \Lambda^{0, \bullet}\otimes L^{p}\otimes E}$ along the curve $u\mapsto uZ, u\in [0,1]$.
We denote by $\Gamma^{L}, \Gamma^{E}$ and $\Gamma^{B, \Lambda^{0, \bullet}\otimes L^{p}\otimes E}$ the corresponding connection form
of $\nabla^{L}, \nabla^{E}$ and $\nabla^{B, \Lambda^{0, \bullet}\otimes L^{p}\otimes E}$ on $B^{X}(x_{0}, \varepsilon)$, respectively.
Then by \cite[(1.2.30)]{Ma07},
\begin{align}\label{2.2a}
\Gamma_{x_{0}}^{\bf \bullet}=0,\ \ \textup{for}\
\Gamma^{\bf \bullet}=\Gamma^{L}, \Gamma^{E}\ \textup{and} \ \Gamma^{B, \Lambda^{0, \bullet}\otimes L^{p}\otimes E}.
\end{align}

\subsection{Taylor expansion of the operator $\mathcal{L}^{t}_{2}$}
Let $dv_{TX}$ be the Riemannian volume form on $(T_{x_{0}}X, g^{T_{x_{0}}X})$.
Let $k(Z)$ be the smooth positive function defined by the equation $dv_{X}(Z)=k(Z)dv_{TX}(Z)$
with $k(0)=1$. Set ${\bf E}=\Lambda^{q}(T^{\ast(0,1)}X)\otimes E$.

Let $s_{L}$ be an unit vector of $L_{x_{0}}$. Using $s_{L}$ and the above trivialization $\S$\ref{s4a},
we get an isometry $E^{q}_{p}\simeq {\bf E}_{x_{0}}$ on $B^{T_{x_{0}}X}(0,\varepsilon)$.
Under our identification, $h^{E^{q}_{p}}$ is $h^{{\bf E}_{x_{0}}}$ on $B^{T_{x_{0}}X}(0,\varepsilon)$.
If $s\in C^{\infty}(T_{x_{0}}X, {\bf E}_{x_{0}})$, then set
\begin{align}\label{2.2}
\big\|s\big\|^{2}_{0,0}=\int_{\mathbb{R}^{2n}}\big|
s(Z)\big|_{h_{x_{0}}^{{\bf E}}}^{2}dv_{TX}(Z).
\end{align}

We denote by $\nabla_{U}$ the ordinary differentiation operator on $T_{x_{0}}X$ in the direction
$U$. If $\alpha=(\alpha_{1}, \cdots, \alpha_{2n})$ is a multi-index, set
$Z^{\alpha}=Z^{\alpha_{1}}_{1}\cdot\cdot\cdot Z^{\alpha_{2n}}_{2n}$. Let $(\partial^{\alpha}R^{L})_{x_{0}}$
be the tensor
$(\partial^{\alpha}R^{L})_{x_{0}}(e_{i}, e_{j})=\partial^{\alpha}\big(R^{L}(e_{i},e_{j})\big)_{x_{0}}$. We denote by
$\mathcal{R}=\sum_{j}Z_{j}e_{j}=Z$ the radical vector field on $\mathbb{R}^{2n}$.
For $s\in C^{\infty}(B^{T_{x_{0}}X}(0,\varepsilon), {\bf E}_{x_{0}})$ and $Z\in B^{T_{x_{0}}X}(0,\varepsilon)$,
for $t=\frac{1}{\sqrt{p}}$, set
\begin{align}\begin{split}\label{2.3}
(\delta_{t}s)(Z)=s(Z/t), \ \ \ \ & \nabla_{t}=\delta_{t}^{-1}tk^{\frac{1}{2}}
\nabla^{B, \Lambda^{0, \bullet}\otimes L^{p}\otimes E}k^{-\frac{1}{2}}\delta_{t},
 \\  \nabla_{0,\bullet}=\nabla_{\bullet}+\frac{1}{2}R^{L}_{x_{0}}(\mathcal{R}, {\sbt}\ ), \ \ \ \ &
\mathcal{L}^{t}_{2}=\delta_{t}^{-1}t^{2}k^{\frac{1}{2}}D^{2}_{p}k^{-\frac{1}{2}}\delta_{t}.
\end{split}\end{align}

It is a consequence of our trivialization that $\mathcal{L}^{t}_{2}$ is self-adjoint
with respect to $\big\|\cdot\big\|_{0,0}$ on
$C^{\infty}_{0}(B^{T_{x_{0}}X}(0,\varepsilon/t), {\bf E}_{x_{0}})$.
We adopt the convention that all tensors will be
evaluated at the base point $x_{0}\in X$, and most of the time, we will omit the subscript $x_{0}$.
Let $\mathcal{L}_{0}, \mathcal{O}'_{1}, \mathcal{O}'_{2}$ be the operators defined as
\cite[(2.5)]{Ma06},\cite[(1.30)]{Ma08}:
\begin{align}\label{2.4}
\mathcal{L}_{0}=&-\sum_{j}(\nabla_{0,e_{j}})^{2}-2n\pi,
\nonumber \\
\mathcal{O}'_{1}(Z)=&-\frac{2}{3}(\partial_{j}R^{L})_{x_{0}}(\mathcal{R}, e_{i})Z_{j}\nabla_{0,e_{i}}
-\frac{1}{3}(\partial_{i}R^{L})_{x_{0}}(\mathcal{R}, e_{i}),
\nonumber \\
\mathcal{O}'_{2}(Z)=&\frac{1}{3}\Big\langle R_{x_{0}}^{TX}(\mathcal{R},e_{i})\mathcal{R}, e_{j}\Big\rangle
\nabla_{0,e_{i}}\nabla_{0,e_{j}}
\\&
+\Big[\frac{2}{3}\Big\langle R_{x_{0}}^{TX}(\mathcal{R},e_{j})e_{j},e_{i}\Big\rangle
-\Big(\frac{1}{2}\sum_{|\alpha|=2}
(\partial^{\alpha} R^{L})_{x_{0}}\frac{Z^{\alpha}}{\alpha!}+R^{E}_{x_{0}}\Big)(\mathcal{R},e_{i})
\Big]\nabla_{0,e_{i}}
\nonumber \\
&-\frac{1}{4}\nabla_{e_{i}}
\Big(\sum_{|\alpha|=2}(\partial^{\alpha} R^{L})_{x_{0}}\frac{Z^{\alpha}}{\alpha!}(\mathcal{R},e_{i})\Big)
-\frac{1}{9}\sum_{i}\Big[\sum_{j}(\partial_{j} R^{L})_{x_{0}}(\mathcal{R},e_{i})Z_{j}\Big]^{2}
\nonumber \\&-\frac{1}{12}\Big[\mathcal{L}_{0},
\Big\langle R_{x_{0}}^{TX}(\mathcal{R},e_{i})\mathcal{R}, e_{i}\Big\rangle_{x_{0}}\Big].
\nonumber \end{align}

Set
\begin{align*}
\Psi=\frac{1}{4}\ ^{c}(dT_{as})_{x_{0}}.
\end{align*}
Then we have the following analogue of \cite[Th.\,2.2]{Ma06}.
\begin{thm}\label{t2.2}
There are second order differential operators $\mathcal{L}^{0}_{2}, \mathcal{O}_{r}(r\geqslant 1)$
which are self-adjoint with respect to $\big\|\cdot \big\|_{0,0}$ on $C_{0}^{\infty}(\mathbb{R}^{2n},
{\bf E}_{x_{0}})$, and
\begin{align}\begin{split}\label{2.6}
\mathcal{L}^{0}_{2}=&\mathcal{L}_{0}-2\omega_{d,x_{0}},
\\ \mathcal{O}_{1}=&\mathcal{O}'_{1}-\pi \sqrt{-1}
\Big\langle \big(\nabla^{B}_{\mathcal{R}}{\bf J}\big)e_{i}, e_{j}\Big\rangle c(e_{i})c(e_{j}),
\\ \mathcal{O}_{2}=&\mathcal{O}'_{2}-R^{B, \Lambda^{0, \bullet}}_{x_{0}}(\mathcal{R}, e_{i})\nabla_{0,e_{i}}
-\frac{\pi}{2}\sqrt{-1}\Big\langle \big(\nabla^{B}\nabla^{B}{\bf J}\big)
_{(\mathcal{R}, \mathcal{R}),x_{0}}e_{i},e_{j}\Big\rangle c(e_{i})c(e_{j})
\\&+\frac{1}{2}\Big(R^{E}_{x_{0}}+\frac{1}{2}\textup{Tr}\big[ R^{T^{(1,0)}X}_{x_{0}}\big]\Big)
(e_{i},e_{j})c(e_{i})c(e_{j})+\frac{1}{4}r^{X}_{x_{0}}-\Psi
\end{split}\end{align}
such that
\begin{align}\label{2.7}
\mathcal{L}^{t}_{2}=\mathcal{L}^{0}_{2}+\sum^{\infty}_{r=1}\mathcal{O}_{r}t^{r}.
\end{align}
\end{thm}

\begin{proof}
We carry out our proof along the lines in \cite[Th.\,2.2]{Ma06}. The details involved differ in two aspects.
First, the formula \cite[(2.11)]{Ma06} there should be replaced by
\begin{align}\label{2.8}
\nabla_{t,e_{i}}=\nabla_{e_{i}}&+\Big(\frac{1}{2}R^{L}_{x_{0}}+\frac{t}{3}(\partial_{k}R^{L})_{x_{0}}Z^{k}
+\frac{t^2}{4}\sum_{|\alpha|=2}(\partial^{\alpha}R^{L})_{x_{0}}\frac{Z^{\alpha}}{\alpha!}+
\frac{t^2}{2}R^{E}_{x_{0}}+\frac{t^2}{2}R^{B, \Lambda^{0, \bullet}}_{x_{0}}\Big)(\mathcal{R},e_{i})\nonumber
\\&-\frac{t^2}{6}\Big\langle R^{TX}_{x_{0}}(e_{i},e_{j})\mathcal{R}, e_{j} \Big\rangle
+O(t^{3}),
\end{align}
which differs from \cite[(2.11)]{Ma06} by the term $R^{B, \Lambda^{0, \bullet}}_{x_{0}}$. Secondly, the analogue
of \cite[(2.12)]{Ma06} is
\begin{align}\label{2.8a}
\big[\nabla_{U}^{B, \Lambda^{0,\bullet}}, c(V)\big]=c(\nabla_{U}^{B}V),\ \textup{for}\ U, V\in TX.
\end{align}
This explains how the Bismut connection $\nabla^{B}$ comes into the expressions (\ref{2.6}), while
in \cite[(2.6)]{Ma06} the Levi-Civita connection $\nabla^{TX}$ appeared. The readers are referred to
\cite[Th.\,2.2]{Ma06} for more details.
\end{proof}

\subsection{Explicit expression of the coefficient ${\bf b}_{1}$}\label{s7}
We now introduce the complex coordinates $z=(z_{1}, \cdots,
  z_{n})$ such that $v_{j}=\sqrt{2}\frac{\partial}{\partial z_{j}}$ is an
  orthonormal basis of $T_{x_{0}}^{(1, 0)}X$.
  We identify $z$ to $\sum^{n}_{j=1}z_{j}
  \frac{\partial}{\partial z_{j}}$ and $\overline{z}$ to
  $\sum^{n}_{j=1}\overline{z}_{j}\frac{\partial}{\partial
  \overline{z}_{j}}$ when we consider $z$ and $\overline{z}$ as
  vector fields. Then
  $Z=z+\overline{z}$ and
  \begin{align}\label{2.11}
  \big|\frac{\partial}{\partial z_{j}}\big|^2=\big|\frac{\partial}
  {\partial \overline{z}_{j}}\big|^2=\frac{1}{2},
  \ \
  |z|^{2}=|\overline{z}|^2=\frac{1}{2}\big|Z\big|^{2}.
  \end{align}

Set
\begin{align}\label{2.12}
 \xi=(\overline{z}_{1}, \cdots, \overline{z}_{q}, z_{q+1}, \cdots,
 z_{n}),\ \ \overline{\xi}=(z_{1}, \cdots, z_{q}, \overline{z}_{q+1}, \cdots,
 \overline{z}_{n});
 \end{align}
\noindent By (\ref{1.13}),
\begin{align}\label{2.13} {\bf J}\frac{\partial}{\partial \xi_{j}}=
\sqrt{-1}\frac{\partial}{\partial \xi_{j}}, \ \
{\bf J}\frac{\partial}{\partial \overline{\xi}_{j}}=-\sqrt{-1}
\frac{\partial}{\partial \overline{\xi}_{j}}, \ \ \textrm{for}\
j=1, \cdots, n. \end{align}
We also identify $\xi$ to $\sum^{n}_{j=1}\xi_{j}
  \frac{\partial}{\partial \xi_{j}}$ and $\overline{\xi}$ to
  $\sum^{n}_{j=1}\overline{\xi}_{j}\frac{\partial}{\partial
  \overline{\xi}_{j}}$ when we consider $\xi$ and $\overline{\xi}$ as
  vector fields. Then $\xi+\overline{\xi}=Z=z+\overline{z}$ and
  \begin{align}\label{2.15}
  \big|\frac{\partial}{\partial \xi_{j}}\big|^2=\big|\frac{\partial}
  {\partial \overline{\xi}_{j}}\big|^2=\frac{1}{2},
  \ \
  |\xi|^{2}=|\overline{\xi}|^2=\frac{1}{2}\big|Z\big|^{2}.
  \end{align}

  Set $u_{j}=\sqrt{2}\frac{\partial}{\partial \xi_{j}}$ and
  \begin{align}\label{2.16}
  f_{2j-1}=\frac{1}{\sqrt{2}}(u_{j}+\overline{u}_{j}), \ \ \
  f_{2j}=\frac{\sqrt{-1}}{\sqrt{2}}(u_{j}-\overline{u}_{j}), \ \ j=1,\cdots,n.
  \end{align}
  Then $\{u_{1},\cdots, u_{n}\}$ forms an orthonormal basis of $T^{(1, 0)}_{{\bf J}, x_{0}}X$ and
  $\{f_{1},\cdots, f_{2n}\}$ is an orthonormal basis of $T_{x_{0}}X$.

Recall that the operator $\mathcal{L}_{0}$ is defined in (\ref{2.4}).
It is very useful to rewrite $\mathcal{L}_{0}$ by using the creation and
annihilation operators. Set
\begin{align}\label{2.18}
b_{j}=-2\nabla_{0,\frac{\partial}{\partial \xi_{j}}}=-2\frac{\partial}{\partial\xi_{j}}
+\pi\overline{\xi}_{j},\ \
 b^{+}_{j}=2\nabla_{0,\frac{\partial}{\partial \overline{\xi}_{j}}}
 =2\frac{\partial}{\partial \overline{\xi}_{j}}+\pi \xi_{j} \
 \ \ b=(b_{1},\cdots b_{n}).\end{align}
Then for any polynomial $g(\xi, \overline{\xi})$ on $\xi$ and $\overline{\xi}$,
\begin{align}\begin{split}\label{2.19}
[b_{i}, b^{+}_{j}]=&b_{i}b^{+}_{j}-b^{+}_{j}b_{i}=-4\pi \delta_{ij},\ \
[b_{i}, b_{j}]=[b^{+}_{i}, b^{+}_{j}]=0, \\
[g(\xi, \overline{\xi}), b_{j}]=&2\frac{\partial}{\partial \xi_{j}}g(\xi, \overline{\xi}),
\ \ \ \ \  [g(\xi, \overline{\xi}), b^{+}_{j}]=-2\frac{\partial}{\partial \overline{\xi}_{j}}
g(\xi, \overline{\xi}). \end{split}\end{align}

By (\ref{2.18}) and (\ref{2.19}), one gets
\begin{align}
 \mathcal{L}_{0}=\sum^{n}_{j=1}b_{j}b^{+}_{j}.
  \end{align}
We now restate the following result, see \cite[Th.\,8.2.3]{Ma07}.
 \begin{thm}\label{t2.3}
 The spectrum of the restriction of $\mathcal{L}_{0}$ to
 $L^{2}(\mathbb{R}^{2n})$ is given by
 \begin{align}\label{2.21}
 \textup{Spec}(\mathcal{L}_{0}|_{L^2(\mathbb{R}^{2n})})
 =\big\{4\pi\sum^{n}_{j=1}\alpha_{j}\big| \
 \alpha=(\alpha_{1}, \cdots, \alpha_{n})\in \mathbb{N}^{n} \big\}\end{align}
 and an orthogonal basis of the eigenspace of
 $4\pi\sum^{n}_{j=1}\alpha_{j}$ is given by
 \begin{align} \label{2.22}
 b^{\alpha}\big(\xi^{\beta} \textup {exp}
 (-\frac{\pi}{2}\sum^{n}_{j=1}|\xi_{j}|^2)\big), \ \ \  with \ \beta\in
 \mathbb{N}^{n}.
 \end{align}
 \end{thm}

 Let $P^{N}$ denote the orthogonal projection
  from $\big(L^{2}(\mathbb{R}^{2n},{\bf E}_{x_{0}}), \|\cdot\|_{0,0}\big)$ onto
 $N=\textrm{Ker}(\mathcal{L}_{0})$. Let $P^{N}(Z, Z')$ be the smooth kernel of $P^{N}$ with respect to
 $dv_{TX}(Z')$. From (\ref{2.22}), we get
 \begin{align}P^{N}(Z,Z')=\textup{exp}
 \Big[-\frac{\pi}{2}\sum^{n}_{j=1}\big(|\xi_{j}|^2+|\xi'_{j}|^2\big)
 +\pi\sum^{n}_{j=1}\xi_{j}\overline{\xi}'_{j}\Big]
 \end{align} and \begin{align}\label{2.24}
 b^{+}_{j}P^{N}=0,\ \ (b_{j}P^{N})(Z,Z')=2\pi
 (\overline{\xi}_{j}-\overline{\xi}'_{j})P(Z,Z'), \ \ \textrm{for}\ j=1, \cdots, n.
 \end{align}

 Let $\mathcal{P}^{N}$ be the orthogonal projection from $\big(L^2(\mathbb{R}^{2n}, {\bf
 E}_{x_0}), \|\cdot\|_{0,0}\big)$ onto $\textrm{Ker}(\mathcal{L}^0_{2})$, and
 $\mathcal{P}^{N}(Z,Z')$ be its smooth kernel with respect to $dv_{TX}(Z')$.
 Set $P^{N^{\bot}}=\textrm{Id}-P^{N}, \mathcal{P}^{N^{\bot}}=\textrm{Id}-\mathcal{P}^{N}$.
 Since \begin{align}\label{2.25}
 \omega_{d}\big|_{(\det(\overline{W})^{\ast})^{\bot}}\leqslant -2\pi,
 \end{align}
  we find that
 \begin{align}\label{2.26}
 \mathcal{P}^{N}(Z,Z')=P^{N}(Z,Z')I_{\det(\overline{W}^{\ast})\otimes E}.
  \end{align}

It is a consequence of Proposition \ref{t3.1} that the second equality of (\ref{2.6})
reduces to the following formula
\begin{align}\label{2.26b}
\mathcal{O}_{1}=\mathcal{O}'_{1}-8\sqrt{-1}\pi
\Big\langle (\nabla_{\mathcal{R}}^{B}{\bf J})\frac{\partial}{\partial z_{j}},
\frac{\partial}{\partial \overline{z}_{k}}\Big\rangle d\overline{z}_{k}\wedge
i_{\frac{\partial}
{\partial \overline{z}_{j}}}.
\end{align}
\noindent Formula (\ref{2.26b}) differs from \cite[(2.6)]{Ma06} by the presence of $\nabla_{\mathcal{R}}^{B}$
instead of $\nabla_{\mathcal{R}}^{X}$. By repeating the proof of \cite[Th.\,2.3]{Ma06} we obtain the following result.
  \begin{thm}\label{t2.4}
  The following relation holds:
  \begin{align}\label{2.27}
  \mathcal{P}^{N}\mathcal{O}_{1}\mathcal{P}^{N}=0.
  \end{align}
  \end{thm}

Set
 \begin{align}\begin{split}\label{2.31b}
 F_{2}=&(\mathcal{L}_{2}^{0})^{-1}\mathcal{P}^{N^{\bot}}\mathcal{O}_{1}
 (\mathcal{L}_{2}^{0})^{-1}\mathcal{P}^{N^{\bot}}\mathcal{O}_{1}\mathcal{P}^{N}-
 (\mathcal{L}_{2}^{0})^{-1}\mathcal{P}^{N^{\bot}}\mathcal{O}_{2}\mathcal{P}^{N}
 \\&+\mathcal{P}^{N}\mathcal{O}_{1}(\mathcal{L}_{2}^{0})^{-1}\mathcal{P}^{N^{\bot}}
 \mathcal{O}_{1}(\mathcal{L}_{2}^{0})^{-1}\mathcal{P}^{N^{\bot}}-
 \mathcal{P}^{N}\mathcal{O}_{2}(\mathcal{L}_{2}^{0})^{-1}\mathcal{P}^{N^{\bot}}
 \\&+(\mathcal{L}_{2}^{0})^{-1}\mathcal{P}^{N^{\bot}}\mathcal{O}_{1}\mathcal{P}^{N}
 \mathcal{O}_{1}(\mathcal{L}_{2}^{0})^{-1}\mathcal{P}^{N^{\bot}}-
 \mathcal{P}^{N}\mathcal{O}_{1}\mathcal{P}^{N^{\bot}}(\mathcal{L}_{2}^{0})^{-2}
 \mathcal{O}_{1}\mathcal{P}^{N}.
 \end{split}\end{align}
By Theorem \ref{t2.4} and the same argument as in \cite[Sec.\,1.5-1.6]{Ma08}, we
 get
 \begin{align}\label{2.31c}
 {\bf b}_{1}(x_{0})=I_{{\bf E}}\cdot F_{2}(0,0) \cdot
 I_{{\bf E}},
 \end{align}
 where $I_{{\bf E}}$ denotes the orthogonal projection from
 $\Lambda (T^{\ast(0, 1)}X)\otimes E$ onto ${\bf E}$.
 We only need to computer the first two terms and the last two terms in (\ref{2.31b}), since the third and
 fourth term in (\ref{2.31b}) are adjoint of the first two terms by Theorem \ref{t2.2}.

\section{Computations of the coefficient ${\bf b}_{1}$}
In this section we calculate out term by term the expressions (\ref{2.31b}) of ${\bf b}_{1}$ and then prove
Corollary \ref{t1.2}.
\begin{lemma}\label{t2.5a}
For every 2-form  $A$, we have
\begin{align}\label{2.44a}
^{c}(A)\cdot I_{\textup{det}(\overline{W}^{\ast})\otimes E}=&
\bigg[-2A(\frac{\partial}{\partial \xi_{j}},\frac{\partial}{\partial\overline{\xi}_{j}})
+4\sum_{j=1}^{q}\sum_{k=q+1}^{n}
A(\frac{\partial}{\partial\overline{\xi}_{j}},\frac{\partial}{\partial\overline{\xi}_{k}})
d\overline{\xi}_{k}\wedge i_{\frac{\partial}{\partial \xi_{j}}}
\nonumber \\&+
4\sum_{j,k=1}^{q}
A(\frac{\partial}{\partial\overline{\xi}_{j}}, \frac{\partial}{\partial\overline{\xi}_{k}})
i_{\frac{\partial}{\partial \xi_{j}}}i_{\frac{\partial}{\partial \xi_{k}}}+
\sum_{j,k=q+1}^{n}
A(\frac{\partial}{\partial\overline{\xi}_{j}},\frac{\partial}{\partial\overline{\xi}_{k}})d\overline{\xi}_{j}\wedge d\overline{\xi}_{k}\bigg]I_{\textup{det}(\overline{W}^{\ast})\otimes E}.
\end{align}
Moreover, if $A$ is compatible with the complex structure $J$, then
\begin{align}\label{2.44b}
^{c}(A)\cdot I_{\textup{det}(\overline{W}^{\ast})\otimes E}=&
\bigg[-2A(\frac{\partial}{\partial \xi_{j}},\frac{\partial}{\partial\overline{\xi}_{j}})
+4\sum_{j=1}^{q}\sum_{k=q+1}^{n}
A(\frac{\partial}{\partial\overline{\xi}_{j}},\frac{\partial}{\partial\overline{\xi}_{k}})
d\overline{\xi}_{k}\wedge i_{\frac{\partial}{\partial \xi_{j}}}
\bigg]I_{\textup{det}(\overline{W}^{\ast})\otimes E}.
\end{align}
\end{lemma}
\begin{proof}
One easily get the result (\ref{2.44a}) from (\ref{1.16b}).
\end{proof}

If $A'$ is a antisymmetric endomorphism of $TX$, an analogue of (\ref{2.44a}) holds for $A'$, i.e.,
simply replacing $A(\cdot,\cdot)$ in (\ref{2.44a}) by $\big\langle A'\cdot,\cdot\big\rangle$.

\subsection{The terms in ${\bf b}_{1}$ containing the factor $\mathcal{O}_{1}$}
Set $\mathcal{J}=-2\pi\sqrt{-1}{\bf J}$. From (\ref{2.14}) and (\ref{2.18}),
 we find that
\begin{align}\label{2.30}
\mathcal{O}'_{1}(Z)=-\frac{2}{3}\bigg[ \Big\langle(\nabla_{Z}^{X}{\mathcal J})\mathcal{R},
\frac{\partial}{\partial \xi_{j}}\Big\rangle b^{+}_{j}-b_{j}
\Big\langle(\nabla_{Z}^{X}{\mathcal J})\mathcal{R},
\frac{\partial}{\partial \overline{\xi}_{j}}\Big\rangle\bigg]
\end{align}

By (\ref{2.14}), (\ref{2.26}), (\ref{2.26b}), (\ref{2.44b}) and (\ref{2.30}), we know that
\begin{align}\begin{split}\label{2.47}
&\mathcal{O}_{1}\mathcal{P}^{N}(Z,Z')
\\=&
\bigg[-\frac{2\sqrt{-1}}{3}b_{i}b_{j}\Big\langle(\nabla_{\frac{\partial}{\partial\overline{\xi}_{j}}}
^{X}{\bf J})\overline{\xi}', \frac{\partial}{\partial \overline{\xi}_{i}}\Big\rangle-
\frac{4\pi\sqrt{-1}}{3}b_{i}\Big\langle (\nabla^{X}_{\overline{\xi}'}{\bf J})\overline{\xi}',
\frac{\partial}{\partial \overline{\xi}_{i}}\Big\rangle
\\&
-4\sqrt{-1}
\sum_{j=1}^{q}\sum_{k=q+1}^{n}\Big\langle(\nabla_{\frac{\partial}{\partial \overline{\xi}_{m}}}^{B}{\bf J})
\frac{\partial}{\partial \overline{\xi}_{j}}, \frac{\partial}{\partial \overline{\xi}_{k}}\Big\rangle
d\overline{\xi}_{k}\wedge i_{\frac{\partial}{\partial \xi_{j}}}
(b_{m}+2\pi\overline{\xi}'_{m})
\\&
-8\pi\sqrt{-1}
\sum_{j=1}^{q}\sum_{k=q+1}^{n}\Big\langle(\nabla_{\xi}^{B}{\bf J})
\frac{\partial}{\partial \overline{\xi}_{j}}, \frac{\partial}{\partial \overline{\xi}_{k}}\Big\rangle
d\overline{\xi}_{k}\wedge i_{\frac{\partial}{\partial \xi_{j}}}
\bigg]P^{N}(Z, Z')I_{\textup{det} (\overline{W}^{\ast})\otimes E}.
\end{split}\end{align}

From Theorem \ref{t2.3} and (\ref{2.47}),
\begin{align}\begin{split}\label{2.48}
&\big((\mathcal{L}^{0}_{2})^{-1}\mathcal{P}^{N^{\bot}}\mathcal{O}_{1}\mathcal{P}^{N}\big)(Z,Z')
\\=&
-\sqrt{-1}\bigg[\frac{b_{i}b_{j}}{12\pi}\Big\langle(\nabla_{\frac{\partial}{\partial\overline{\xi}_{j}}}
^{X}{\bf J})\overline{\xi}', \frac{\partial}{\partial \overline{\xi}_{i}}\Big\rangle+
\frac{b_{i}}{3}\Big\langle (\nabla^{X}_{\overline{\xi}'}{\bf J})\overline{\xi}',
\frac{\partial}{\partial \overline{\xi}_{i}}\Big\rangle
\\&+
\sum^{q}_{j=1}\sum^{n}_{k=q+1}\Big\langle(\nabla_{\frac{\partial}{\partial \overline{\xi}_{m}}}^{B}{\bf J})
\frac{\partial}{\partial \overline{\xi}_{j}}, \frac{\partial}{\partial \overline{\xi}_{k}}\Big\rangle
d\overline{\xi}_{k}\wedge i_{\frac{\partial}{\partial \xi_{j}}}
(\frac{b_{m}}{3\pi}+\overline{\xi}'_{m})
\\&+
\sum^{q}_{j=1}\sum^{n}_{k=q+1}\Big\langle(\nabla_{\xi}^{B}{\bf J})
\frac{\partial}{\partial \overline{\xi}_{j}}, \frac{\partial}{\partial \overline{\xi}_{k}}\Big\rangle
d\overline{\xi}_{k}\wedge i_{\frac{\partial}{\partial \xi_{j}}}
\bigg]P^{N}(Z, Z')I_{\textup{det} (\overline{W}^{\ast})\otimes E}.\end{split}
\end{align}
Therefore,
\begin{align}\begin{split}\label{2.49}
 &\Big((\mathcal{L}^{0}_{2})^{-1}\mathcal{P}^{N^{\bot}}\mathcal{O}_{1}\mathcal{P}^{N}\Big)(0,Z')
\\ =& -\frac{\sqrt{-1}}{3}\sum^{q}_{j=1}\sum^{n}_{k=q+1}
\Big\langle(\nabla_{\overline{\xi}'}^{B}{\bf J})
\frac{\partial}{\partial \overline{\xi}_{j}}, \frac{\partial}{\partial \overline{\xi}_{k}}\Big\rangle
d\overline{\xi}_{k}\wedge i_{\frac{\partial}{\partial \xi_{j}}}
P^{N}(0, Z')I_{\textup{det} (\overline{W}^{\ast})\otimes E}.\end{split}
\end{align}
and
\begin{align}\begin{split}\label{2.50}
&\Big((\mathcal{L}^{0}_{2})^{-1}\mathcal{P}^{N^{\bot}}\mathcal{O}_{1}\mathcal{P}^{N}\Big)(Z,0)
\\ =&-\frac{2\sqrt{-1}}{3}\sum^{q}_{j=1}\sum^{n}_{k=q+1}
\Big\langle(\nabla_{\overline{\xi}}^{B}{\bf J})
\frac{\partial}{\partial \overline{\xi}_{j}}, \frac{\partial}{\partial \overline{\xi}_{k}}\Big\rangle
d\overline{\xi}_{k}\wedge i_{\frac{\partial}{\partial \xi_{j}}}
P^{N}(Z, 0)I_{\textup{det} (\overline{W}^{\ast})\otimes E}
\\ &-\sqrt{-1}\sum^{q}_{j=1}\sum^{n}_{k=q+1}
\Big\langle(\nabla_{\xi}^{B}{\bf J})
\frac{\partial}{\partial \overline{\xi}_{j}}, \frac{\partial}{\partial \overline{\xi}_{k}}\Big\rangle
d\overline{\xi}_{k}\wedge i_{\frac{\partial}{\partial \xi_{j}}}
P^{N}(Z, 0)I_{\textup{det} (\overline{W}^{\ast})\otimes E}.
\end{split}\end{align}

By taking adjoint of (\ref{2.49}) and (\ref{2.50}), we find that
\begin{align}\begin{split}\label{2.51}
&\Big(\mathcal{P}^{N}\mathcal{O}_{1}(\mathcal{L}^{0}_{2})^{-1}\mathcal{P}^{N^{\bot}}\Big)(Z',0)
\\=& \frac{\sqrt{-1}}{3}\sum^{q}_{j=1}\sum^{n}_{k=q+1}
\Big\langle(\nabla_{\xi'}^{B}{\bf J})
\frac{\partial}{\partial \xi_{j}}, \frac{\partial}{\partial \xi_{k}}\Big\rangle
P^{N}(Z',0) I_{\textup{det} (\overline{W}^{\ast})\otimes E}
d\xi_{j}\wedge i_{\frac{\partial}{\partial \overline{\xi}_{k}}},
\end{split}\end{align}
and
\begin{align}\begin{split}\label{2.52}
& \Big(\mathcal{P}^{N}\mathcal{O}_{1}(\mathcal{L}^{0}_{2})^{-1}\mathcal{P}^{N^{\bot}}\Big)(0,Z)
\\=&\frac{2\sqrt{-1}}{3}\sum^{q}_{j=1}\sum^{n}_{k=q+1}
\Big\langle(\nabla_{\xi}^{B}{\bf J})
\frac{\partial}{\partial \xi_{j}}, \frac{\partial}{\partial \xi_{k}}\Big\rangle
P^{N}(0,Z)I_{\textup{det} (\overline{W}^{\ast})\otimes E}
d\xi_{j}\wedge i_{\frac{\partial}{\partial \overline{\xi}_{k}}}
\\&+\sqrt{-1}\sum^{q}_{j=1}\sum^{n}_{k=q+1}
\Big\langle(\nabla_{\overline{\xi}}^{B}{\bf J})
\frac{\partial}{\partial \xi_{j}}, \frac{\partial}{\partial \xi_{k}}\Big\rangle
P^{N}(0,Z)I_{\textup{det} (\overline{W}^{\ast})\otimes E}
d\xi_{j}\wedge i_{\frac{\partial}{\partial \overline{\xi}_{k}}}.
\end{split}\end{align}

By (\ref{3.8}), (\ref{2.50}),
(\ref{2.52}) and $\int_{\mathbb{C}}|\xi|^{2}e^{-\pi|\xi|^{2}}=\frac{1}{\pi}$,
we obtain  \begin{align}\begin{split}\label{2.53}
&\Big(\mathcal{P}^{N}\mathcal{O}_{1}\mathcal{P}^{N^{\bot}}
(\mathcal{L}^{0}_{2})^{-2}\mathcal{O}_{1}\mathcal{P}^{N}\Big)(0,0)
\\=&
\frac{4}{9\pi}\sum^{q}_{j=1}\sum^{n}_{k=q+1}
\Big|\big\langle(\nabla_{\frac{\partial}{\partial \xi_{m}}}^{B}{\bf J})
\frac{\partial}{\partial \xi_{j}}, \frac{\partial}{\partial \xi_{k}}\big\rangle\Big|^{2}
I_{\textup{det} (\overline{W}^{\ast})\otimes E}
\\&
+\frac{1}{\pi}\sum^{q}_{j=1}\sum^{n}_{k=q+1}
\Big|\big\langle(\nabla_{\frac{\partial}{\partial \overline{\xi}_{m}}}^{B}{\bf J})
\frac{\partial}{\partial \xi_{j}}, \frac{\partial}{\partial \xi_{k}}\big\rangle\Big|^{2}
I_{\textup{det} (\overline{W}^{\ast})\otimes E}
\\=&
\frac{1}{72\pi}\Big(\big|\nabla^{B}{\bf J}\big|^{2}+
10\big|\big\langle S^{B}(\overline{u}_{i})u_{j}, u_{k}\big\rangle \big|^{2}\Big)
I_{\textup{det} (\overline{W}^{\ast})\otimes E}.
\end{split}\end{align}

By (\ref{2.49}) and (\ref{2.51}),
\begin{align}\begin{split}\label{2.54}
& \Big((\mathcal{L}^{0}_{2})^{-1}\mathcal{P}^{N^{\bot}}\mathcal{O}_{1}\mathcal{P}^{N}\mathcal{O}_{1}
(\mathcal{L}^{0}_{2})^{-1}\mathcal{P}^{N^{\bot}}\Big)(0,0).
\\=&\frac{1}{9\pi}\sum^{q}_{i,j=1}\sum^{n}_{k,l=q+1}
\Big\langle(\nabla_{\frac{\partial}{\partial \overline{\xi}_{m}}}^{B}{\bf J})
\frac{\partial}{\partial \overline{\xi}_{i}}, \frac{\partial}{\partial \overline{\xi}_{l}}\Big\rangle
\Big\langle(\nabla_{\frac{\partial}{\partial \xi_{m}}}^{B}{\bf J})
\frac{\partial}{\partial \xi_{j}}, \frac{\partial}{\partial \xi_{k}}\Big\rangle
d\overline{\xi}_{l}\wedge i_{\frac{\partial}{\partial \xi_{i}}}
I_{\textup{det} (\overline{W}^{\ast})\otimes E}d\xi_{j}\wedge i_{\frac{\partial}{\partial \overline{\xi}_{k}}}.
\end{split}\end{align}

Let $h(Z)$ (resp. $F(Z)$) be homogenous polynomials in $Z$ with degree $1$ (resp. $2$), then
by (\ref{2.24}) and Theorem \ref{t2.3},
\begin{align}\begin{split}\label{2.56}
\big(\mathcal{L}^{-1}_{0}P^{N^\bot}h b_{j}P^{N}\big)(0,0)=&
\big(\mathcal{L}^{-1}_{0}P^{N^\bot}b_{j}hP^{N}\big)(0,0)=-\frac{1}{2\pi}\frac{\partial h}{\partial \xi_{j}},\\
\big(\mathcal{L}^{-1}_{0}P^{N^\bot}FP^{N}\big)(0,0)=&-\frac{1}{4\pi^{2}}\frac{\partial^2 F}{\partial \xi_{j}
\partial \overline{\xi}_{j}}.\end{split}
\end{align}

\noindent
From (\ref{2.24}) and Theorem \ref{t2.3},
one verifies directly that for $1\leqslant j\leqslant q, q+1\leqslant k\leqslant n$,
\begin{align}\begin{split}\label{2.59}
\Big( (\mathcal{L}^{0}_{2})^{-1}h b_{m}d\overline{\xi}_{k}\wedge
i_{\frac{\partial}{\partial \xi_{j}}}
\mathcal{P}^{N}\Big)(0,0)&=\frac{1}{12\pi }\frac{\partial h}
{\partial \xi_{m}}d\overline{\xi}_{k}\wedge
i_{\frac{\partial}{\partial \xi_{j}}}I_{\textup{det}({\overline W}^{\ast})\otimes E},
\\ \Big( (\mathcal{L}^{0}_{2})^{-1}F d\overline{\xi}_{k}\wedge
i_{\frac{\partial}{\partial \xi_{j}}}
\mathcal{P}^{N}\Big)(0,0)&=\frac{1}{24\pi^{2}}\frac{\partial^2 F}
{\partial \xi_{i} \partial\overline{\xi}_{i}}d\overline{\xi}_{k}\wedge
i_{\frac{\partial}{\partial \xi_{j}}}I_{\textup{det}({\overline W}^{\ast})\otimes E}
.\end{split}
\end{align}

\noindent
Moreover, we have for $1\leqslant i, j\leqslant q$ and $q+1\leqslant k, l\leqslant n$,
\begin{align}\label{2.59a}
 \Big( (\mathcal{L}^{0}_{2})^{-1}F d\overline{\xi}_{k}\wedge d\overline{\xi}_{l}
i_{\frac{\partial}{\partial \xi_{i}}}i_{\frac{\partial}{\partial \xi_{j}}}
\mathcal{P}^{N}\Big)(0,0)=\frac{1}{80\pi^{2}}\frac{\partial^2 F}
{\partial \xi_{m} \partial\overline{\xi}_{m}}d\overline{\xi}_{k}\wedge d\overline{\xi}_{l}
i_{\frac{\partial}{\partial \xi_{i}}}
i_{\frac{\partial}{\partial \xi_{j}}}I_{\textup{det}({\overline W}^{\ast})\otimes E}.
\end{align}

Using (\ref{2.14}), (\ref{2.24}), (\ref{2.30}) and (\ref{2.50}) we obtain
\begin{align}\label{2.59b}
\Big( (\mathcal{L}^{0}_{2})^{-1}\mathcal{P}^{N^{\bot}}\mathcal{O}'_{1}(\mathcal{L}^{0}_{2})^{-1}
\mathcal{O}_{1}\mathcal{P}^{N}\Big)(0,0)=0.
\end{align}

\noindent
Then by (\ref{3.8}), (\ref{2.26b}), (\ref{2.50}), (\ref{2.59a}) and (\ref{2.59b}) we get
\begin{align}\label{2.60}
&\Big( (\mathcal{L}^{0}_{2})^{-1}\mathcal{P}^{N^{\bot}}\mathcal{O}_{1}(\mathcal{L}^{0}_{2})^{-1}
\mathcal{O}_{1}\mathcal{P}^{N}\Big)(0,0)
\nonumber \\=&
-\frac{16}{3}\pi\bigg\{
(\mathcal{L}^{0}_{2})^{-1}\mathcal{P}^{N^{\bot}}
\Big[\big\langle (\nabla_{\mathcal{R}}^{B}{\bf J})\frac{\partial}{\partial z_{i}},
\frac{\partial}{\partial \overline{z}_{l}}\Big\rangle d\overline{z}_{l}\wedge
i_{\frac{\partial}
{\partial \overline{z}_{i}}}
\nonumber \\&
\times \sum^{q}_{j=1}\sum^{n}_{k=q+1}\Big\langle(\nabla_{\overline{\xi}}^{B}{\bf J})
\frac{\partial}{\partial \overline{\xi}_{j}}, \frac{\partial}{\partial \overline{\xi}_{k}}\Big\rangle
d\overline{\xi}_{k}\wedge i_{\frac{\partial}{\partial \xi_{j}}}\Big]
P^{N}\bigg\}(0,0)I_{\textup{det}(\overline{W}^{\ast})\otimes E}
\nonumber \\&
-8\pi\bigg\{
(\mathcal{L}^{0}_{2})^{-1}\mathcal{P}^{N^{\bot}}
\Big[\big\langle (\nabla_{\mathcal{R}}^{B}{\bf J})\frac{\partial}{\partial z_{i}},
\frac{\partial}{\partial \overline{z}_{l}}\Big\rangle d\overline{z}_{l}\wedge
i_{\frac{\partial}
{\partial \overline{z}_{i}}}
\nonumber \\&
\times \sum^{q}_{j=1}\sum^{n}_{k=q+1}\Big\langle(\nabla_{\xi}^{B}{\bf J})
\frac{\partial}{\partial \overline{\xi}_{j}}, \frac{\partial}{\partial \overline{\xi}_{k}}\Big\rangle
d\overline{\xi}_{k}\wedge i_{\frac{\partial}{\partial \xi_{j}}}\Big]
P^{N}\bigg\}(0,0)I_{\textup{det}(\overline{W}^{\ast})\otimes E}
 \\=&
-\frac{4}{3\pi}\sum_{j=1}^{q}\sum_{k=q+1}^{n}
\Big|\big\langle \big(\nabla_{\frac{\partial}{\partial \xi_{m}}}^{B}{\bf J}\big)
\frac{\partial}{\partial \xi_{j}}, \frac{\partial}{\partial \xi_{k}}\big\rangle\Big|^{2}
I_{\textup{det}({\overline W}^{\ast})\otimes E}+I_{1}
\nonumber \\&
-\frac{2}{\pi}\sum_{j=1}^{q}\sum_{k=q+1}^{n}
\Big|\big\langle \big(\nabla_{\frac{\partial}{\partial \overline{\xi}_{m}}}^{B}{\bf J}\big)
\frac{\partial}{\partial \xi_{j}}, \frac{\partial}{\partial \xi_{k}}\big\rangle\Big|^{2}
I_{\textup{det}({\overline W}^{\ast})\otimes E}+I_{2}
\nonumber \\=&
-\frac{1}{24\pi}\Big(\big|\nabla^{B}{\bf J}\big|^{2}+
4\big|\big\langle S^{B}(\overline{u}_{i})u_{j}, u_{k}\big\rangle \big|^{2}\Big)
I_{\textup{det}({\overline W}^{\ast})\otimes E}+I_{1}+I_{2},
\nonumber
\end{align}
with
\begin{align}\begin{split}
I_{1}=-\frac{1}{15\pi}\sum^{q}_{i,j=1}\sum^{n}_{k,l=q+1}
&\Big\langle(\nabla_{\frac{\partial}{\partial \xi_{m}}}^{B}{\bf J})
\frac{\partial}{\partial \overline{\xi}_{i}}, \frac{\partial}{\partial \overline{\xi}_{l}}\Big\rangle
\Big\langle(\nabla_{\frac{\partial}{\partial \overline{\xi}_{m}}}^{B}{\bf J})
\frac{\partial}{\partial \overline{\xi}_{j}}, \frac{\partial}{\partial \overline{\xi}_{k}}\Big\rangle
\\
\  & \times d\overline{\xi}_{k}\wedge d\overline{\xi}_{l} \wedge i_{\frac{\partial}{\partial \xi_{i}}}
i_{\frac{\partial}{\partial \xi_{j}}}I_{\textup{det} (\overline{W}^{\ast})\otimes E}
\end{split}\end{align}
and
\begin{align}\begin{split}
I_{2}=-\frac{1}{10\pi}\sum^{q}_{i,j=1}\sum^{n}_{k,l=q+1}
&\Big\langle(\nabla_{\frac{\partial}{\partial \overline{\xi}_{m}}}^{B}{\bf J})
\frac{\partial}{\partial \overline{\xi}_{i}}, \frac{\partial}{\partial \overline{\xi}_{l}}\Big\rangle
\Big\langle(\nabla_{\frac{\partial}{\partial \xi_{m}}}^{B}{\bf J})
\frac{\partial}{\partial \overline{\xi}_{j}}, \frac{\partial}{\partial \overline{\xi}_{k}}\Big\rangle
\\
\  & \times d\overline{\xi}_{k}\wedge d\overline{\xi}_{l}\wedge i_{\frac{\partial}{\partial \xi_{i}}}
i_{\frac{\partial}{\partial \xi_{j}}}I_{\textup{det} (\overline{W}^{\ast})\otimes E}.
\end{split}\end{align}

\subsection{The terms in ${\bf b}_{1}$ containing the factor $\mathcal{O}_{2}$}
Before calculating the term \\ $\big((\mathcal{L}^{0}_{2})^{-1}\mathcal{P}^{N^{\bot}}\mathcal{O}_{2}\mathcal{P}^{N}\big)(0,0)$
in (\ref{2.31b}), we first need to derive a formula for
\\ $\big({\mathcal L}_{0}^{-1}{P}^{N^{\bot}}\mathcal{O}'_{2}P^{N}\big)(0,0)$.

\begin{lemma}\label{t2.7}
\begin{align}\label{2.63}
-\Big(\mathcal{L}^{-1}_{0}P^{N^{\bot}}\mathcal{O}'_{2}P^{N}\Big)(0,0)=
\frac{1}{2\pi} \Big[R^{E}(\frac{\partial}{\partial \xi_{i}}, \frac{\partial}{\partial\overline{\xi}_{i}})
+\Big\langle R^{TX}(\frac{\partial}{\partial \xi_{i}}, \frac{\partial}{\partial\overline{\xi}_{j}})
\frac{\partial}{\partial \xi_{j}}, \frac{\partial}{\partial\overline{\xi}_{i}}\Big\rangle\Big].
\end{align}
\end{lemma}

\noindent
This result is the analogue of \cite[(2.39)]{Ma08}). In our case we use the coordinates
$(\xi_{1}, \ldots, \xi_{n})$ adapted to ${\bf J}$ (see (\ref{2.13})), whereas in \cite[(2.39)]{Ma08}) the
coordinates $(z_{1}, \ldots, z_{n})$ adapted to $J$ are used (see \cite[\S 1.4]{Ma08}). The proof of
(\ref{2.63}) is similar to \cite[(2.39)]{Ma08} so we omit it.

Now we are ready to compute the term $\Big((\mathcal{L}^{0}_{2})^{-1}\mathcal{P}^{N^{\bot}}\Psi \mathcal{P}^{N}\Big)(0,0)$.
Note first that $\Psi=\frac{1}{4}(d T_{as})\in \Lambda^{2}(T^{\ast(1,0)}X)\otimes \Lambda^{2}(T^{\ast(0,1)}X)$.
By (\ref{1.16a}) we obtain
\begin{align}
24\Psi=&\frac{1}{4}\big(dT_{as}\big)(e_{i}, e_{j}, e_{k}, e_{k})c(e_{i})c(e_{j})c(e_{k})c(e_{l})
\nonumber \\=& -\frac{1}{2}\big(dT_{as}\big)(e_{i}, e_{j}, v_{k}, \overline{v}_{k})c(e_{i})c(e_{j})
+\big(dT_{as}\big)(e_{i}, e_{j}, v_{k}, \overline{v}_{l})c(e_{i})c(e_{j})\overline{v}^{l}\wedge i_{\overline{v}_{k}}
\nonumber \\& \ \
+\frac{1}{2}\big(dT_{as}\big)(e_{i}, e_{j}, v_{k}, v_{l})c(e_{i})c(e_{j})i_{\overline{v}_{k}}i_{\overline{v}_{l}}
+\frac{1}{2}\big(dT_{as}\big)(e_{i}, e_{j}, \overline{v}_{k}, \overline{v}_{l})c(e_{i})c(e_{j})
\overline{v}^{k}\wedge \overline{v}^{l}
\nonumber \\=&-2\Big[-\frac{1}{2}\big(dT_{as}\big)(v_{j}, \overline{v}_{j}, v_{k}, \overline{v}_{k})
+\big(dT_{as}\big)(v_{i}, \overline{v}_{j}, v_{k}, \overline{v}_{k})\overline{v}^{j}\wedge i_{\overline{v}_{i}}\Big]
\nonumber \\ &
+4\Big[-\frac{1}{2}\big(dT_{as}\big)(v_{j}, \overline{v}_{j}, v_{k}, \overline{v}_{l})
+\big(dT_{as}\big)(v_{i}, \overline{v}_{j}, v_{k}, \overline{v}_{l})\overline{v}^{j}
\wedge i_{\overline{v}_{i}}\Big]\overline{v}^{l}\wedge i_{\overline{v}_{k}}
\nonumber \\ &
+\big(dT_{as}\big)(\overline{v}_{i}, \overline{v}_{j}, v_{k}, v_{l})\overline{v}^{i}\wedge
\overline{v}^{j}\wedge i_{\overline{v}_{k}}i_{\overline{v}_{l}}
+\big(dT_{as}\big)(v_{i}, v_{j}, \overline{v}_{k}, \overline{v}_{l})i_{\overline{v}_{i}}i_{\overline{v}_{j}}
\overline{v}^{k}\wedge\overline{v}^{l}
\nonumber \\=&\big(dT_{as}\big)(v_{j}, \overline{v}_{j}, v_{k}, \overline{v}_{k})
-4\big(dT_{as}\big)(v_{i}, \overline{v}_{j}, v_{k}, \overline{v}_{k})
\overline{v}^{j}\wedge i_{\overline{v}_{i}}
\nonumber \\ &
+4\big(dT_{as}\big)(v_{i}, \overline{v}_{j}, v_{k}, \overline{v}_{l})\overline{v}^{j}
\wedge i_{\overline{v}_{i}}\overline{v}^{l}\wedge i_{\overline{v}_{k}}
\nonumber \\ &
+\big(dT_{as}\big)(\overline{v}_{i}, \overline{v}_{j}, v_{k}, v_{l})\overline{v}^{i}\wedge
\overline{v}^{j}\wedge i_{\overline{v}_{k}}i_{\overline{v}_{l}}
+\big(dT_{as}\big)(v_{i}, v_{j}, \overline{v}_{k}, \overline{v}_{l})i_{\overline{v}_{i}}i_{\overline{v}_{j}}
\overline{v}^{k}\wedge\overline{v}^{l}
\nonumber.
\end{align}
Using the relation
\begin{align}
\overline{v}^{i}\wedge i_{\overline{v}_{j}}+i_{\overline{v}_{j}}\overline{v}^{i}=\delta_{ij},
\end{align}
we get
\begin{align}\begin{split}
24\Psi=&3\big(dT_{as}\big)(v_{i}, \overline{v}_{i}, v_{j}, \overline{v}_{j})
-12\big(dT_{as}\big)(v_{i}, \overline{v}_{j}, v_{k}, \overline{v}_{k})
\overline{v}^{j}\wedge i_{\overline{v}_{i}}
\\&+
6\big(dT_{as}\big)(v_{i}, v_{j}, \overline{v}_{k}, \overline{v}_{l})
\overline{v}^{k}\wedge
\overline{v}^{l}\wedge i_{\overline{v}_{i}}i_{\overline{v}_{j}}.
\end{split}\end{align}
That is
\begin{align}\begin{split}
\Psi=&\frac{1}{2}\big(dT_{as}\big)(\frac{\partial}{\partial z_{i}},
\frac{\partial}{\partial \overline{z}_{i}}, \frac{\partial}{\partial z_{j}},
\frac{\partial}{\partial \overline{z}_{j}})
-2\big(dT_{as}\big)(\frac{\partial}{\partial z_{i}}, \frac{\partial}{\partial \overline{z}_{j}},
\frac{\partial}{\partial z_{k}}, \frac{\partial}{\partial \overline{z}_{k}})
d\overline{z}_{j}\wedge i_{\frac{\partial}{\partial \overline{z}_{i}}}
\\&+
\big(dT_{as}\big)(\frac{\partial}{\partial z_{i}}, \frac{\partial}{\partial z_{j}},
\frac{\partial}{\partial \overline{z}_{k}}, \frac{\partial}{\partial \overline{z}_{l}})
d\overline{z}_{k}\wedge
d\overline{z}_{l}\wedge i_{\frac{\partial}{\partial \overline{z}_{i}}}
i_{\frac{\partial}{\partial \overline{z}_{j}}}.
\end{split}\end{align}

Hence,
\begin{align}
\Psi I_{\textup{det}(\overline{W}^{\ast})\otimes E}
=&\frac{1}{2}\big(dT_{as}\big)(\frac{\partial}{\partial z_{i}},
\frac{\partial}{\partial \overline{z}_{i}}, \frac{\partial}{\partial z_{j}},
\frac{\partial}{\partial \overline{z}_{j}})I_{\textup{det}(\overline{W}^{\ast})\otimes E}
\nonumber \\&
-2\sum_{i=1}^{q}\big(dT_{as}\big)(\frac{\partial}{\partial z_{i}},
\frac{\partial}{\partial \overline{z}_{i}}, \frac{\partial}{\partial z_{k}},
\frac{\partial}{\partial \overline{z}_{k}})I_{\textup{det}(\overline{W}^{\ast})\otimes E}
\nonumber \\&
-2\sum_{i=1}^{q}\sum_{j=q+1}^{n}\big(dT_{as}\big)(\frac{\partial}{\partial z_{i}},
\frac{\partial}{\partial \overline{z}_{j}}, \frac{\partial}{\partial z_{k}},
\frac{\partial}{\partial \overline{z}_{k}})d\overline{z}_{j}
\wedge i_{\frac{\partial}{\partial \overline{z}_{i}}}
I_{\textup{det}(\overline{W}^{\ast})\otimes E}
\nonumber \\&
+2\sum_{i, j=1}^{q}\big(dT_{as}\big)(\frac{\partial}{\partial z_{i}},
\frac{\partial}{\partial z_{j}}, \frac{\partial}{\partial \overline{z}_{j}},
\frac{\partial}{\partial \overline{z}_{i}})
I_{\textup{det}(\overline{W}^{\ast})\otimes E}
\nonumber \\&
+4\sum_{i, j=1}^{q}\sum_{k=q+1}^{n}\big(dT_{as}\big)(\frac{\partial}{\partial z_{i}},
\frac{\partial}{\partial z_{j}}, \frac{\partial}{\partial \overline{z}_{k}},
\frac{\partial}{\partial \overline{z}_{i}})d\overline{z}_{k}
\wedge i_{\frac{\partial}{\partial \overline{\xi}_{j}}}
I_{\textup{det}(\overline{W}^{\ast})\otimes E}
\nonumber \\&
+\sum_{i, j=1}^{q}\sum_{k, l=q+1}^{n}\big(dT_{as}\big)(\frac{\partial}{\partial z_{i}},
\frac{\partial}{\partial z_{j}}, \frac{\partial}{\partial \overline{z}_{k}},
\frac{\partial}{\partial \overline{z}_{l}})d\overline{z}_{k}
\wedge d\overline{z}_{l}\wedge i_{\frac{\partial}{\partial \overline{\xi}_{i}}}
i_{\frac{\partial}{\partial \overline{\xi}_{j}}}
I_{\textup{det}(\overline{W}^{\ast})\otimes E}
\nonumber \\=&
\frac{1}{2}\big(dT_{as}\big)(\frac{\partial}{\partial \xi_{i}},
\frac{\partial}{\partial \overline{\xi}_{i}}, \frac{\partial}{\partial \xi_{j}},
\frac{\partial}{\partial \overline{\xi}_{j}})I_{\textup{det}(\overline{W}^{\ast})\otimes E}
\nonumber \\&
-2\sum_{i=1}^{q}\sum_{j=q+1}^{n}\big(dT_{as}\big)(\frac{\partial}{\partial \overline{\xi}_{i}},
\frac{\partial}{\partial \overline{\xi}_{j}}, \frac{\partial}{\partial \xi_{k}},
\frac{\partial}{\partial \overline{\xi}_{k}})d\overline{\xi}_{j}
\wedge i_{\frac{\partial}{\partial \xi_{i}}}
I_{\textup{det}(\overline{W}^{\ast})\otimes E}
\nonumber \\&
+\sum_{i, j=1}^{q}\sum_{k, l=q+1}^{n}\big(dT_{as}\big)(\frac{\partial}{\partial \overline{\xi}_{i}},
\frac{\partial}{\partial \overline{\xi}_{j}}, \frac{\partial}{\partial \overline{\xi}_{k}},
\frac{\partial}{\partial \overline{\xi}_{l}})d\overline{\xi}_{k}
\wedge d\overline{\xi}_{l}\wedge i_{\frac{\partial}{\partial \xi_{i}}}
i_{\frac{\partial}{\partial \xi_{j}}}
I_{\textup{det}(\overline{W}^{\ast})\otimes E}.
\end{align}

Now it follows immediately that
\begin{align}\begin{split}\label{2.80}
&\Big((\mathcal{L}^{0}_{2})^{-1}\mathcal{P}^{N^{\bot}}\Psi \mathcal{P}^{N}\Big)(0,0)
\\=&
\frac{1}{16\pi}\Big[(dT_{as})(\frac{\partial}{\partial \overline{\xi}_{i}},
\frac{\partial}{\partial \overline{\xi}_{j}}, \frac{\partial}{\partial \overline{\xi}_{k}},
\frac{\partial}{\partial \overline{\xi}_{l}})d\overline{\xi}_{k}\wedge d\overline{\xi}_{l}\wedge i_{\frac{\partial}{\partial\xi_{i}}}i_{\frac{\partial}{\partial\xi_{j}}}
\\&-4\sum_{j=1}^{q}\sum_{k=q+1}^{n}(dT_{as})(\frac{\partial}{\partial \xi_{i}},
\frac{\partial}{\partial \overline{\xi}_{i}}, \frac{\partial}{\partial \overline{\xi}_{j}},
\frac{\partial}{\partial \overline{\xi}_{k}})
d\overline{\xi}_{k}\wedge i_{\frac{\partial}{\partial\xi_{j}}}
\Big] I_{\textup{det}(\overline{W}^{\ast})\otimes E}.
\end{split}\end{align}

By (\ref{2.6}), we get
\begin{align}\label{2.96}
\mathcal{O}_{2}=&\mathcal{O}'_{2}+R_{x_{0}}^{B, \Lambda^{0, \bullet}}(\mathcal{R}, \frac{\partial}{\partial \overline{\xi}_{j}})b_{j}
-R_{x_{0}}^{B, \Lambda^{0, \bullet}}(\mathcal{R}, \frac{\partial}{\partial \xi_{j}})b^{+}_{j}
\nonumber \\&
-\frac{\sqrt{-1}}{2}\pi \Big\langle \big(\nabla^{B}\nabla^{B}{\bf J}\big)_{(\mathcal{R},\mathcal{R})}
e_{l},e_{m}\Big\rangle c(e_{l})c(e_{m})
\\&
+\frac{1}{2}
\Big(R^{E}_{x_{0}}+\frac{1}{2}\textup{Tr}\big[R^{T^{(1,0)}X}\big]\Big)(e_{l},e_{m})
c(e_{l})c(e_{m})+\frac{1}{4}r^{X}_{x_{0}}-\Psi.
\nonumber \end{align}

From (\ref{1.16b}), (\ref{7.3}), (\ref{2.44b}) and Proposition \ref{t3.1}, we obtain
\begin{align}\label{2.97}
&\mathcal{P}^{N^{\bot}}\big(\mathcal{O}_{2}-\mathcal{O}'_{2}+\Psi\big)\mathcal{P}^{N}
\nonumber \\=&
\mathcal{P}^{N^{\bot}}\Bigg\{\frac{1}{2}\textup{Tr}[R^{T^{(1,0)}X}]\big(\mathcal{R},\frac{\partial}
{\partial \overline{\xi}_{i}}\big)b_{i}
\nonumber \\&
-\bigg\langle\Big(R^{B}(\mathcal{R},\frac{\partial}
{\partial \overline{\xi}_{i}})b_{i}-2\pi\sqrt{-1}(\nabla^{B}\nabla^{B}
{\bf J})_{(\mathcal{R}, \mathcal{R})}\Big)\frac{\partial}
{\partial \xi_{j}}, \frac{\partial}
{\partial \overline{\xi}_{j}} \bigg\rangle
 \\&
+2\sum_{j=1}^{q}\sum_{k=q+1}^{n}\bigg[\Big\langle
\Big(2R^{B}(\mathcal{R}, \frac{\partial}
{\partial \overline{\xi}_{i}})b_{i}-2\sqrt{-1}\pi
(\nabla^{B}\nabla^{B}{\bf J})_{(\mathcal{R},\mathcal{R})}\Big)\frac{\partial}
{\partial \overline{\xi}_{j}}, \frac{\partial}
{\partial \overline{\xi}_{k}}\Big\rangle
\nonumber \\  &
+2\Big(R^{E}+\frac{1}{2}\textup{Tr}[R^{T^{(1,0)}X}]\Big)
\big(\frac{\partial}
{\partial \overline{\xi}_{j}}, \frac{\partial}
{\partial \overline{\xi}_{k}}\big)\bigg]d\overline{\xi}_{k}\wedge
i_{\frac{\partial}{\partial \xi_{j}}}\Bigg\}\mathcal{P}^{N}.
\nonumber
\end{align}

Using (\ref{1.15}), (\ref{2.24}), (\ref{2.56})$-$(\ref{2.59}) and Theorem \ref{t2.3}, we have
\begin{align} \label{2.98}
&-\Big((\mathcal{L}^{0}_{2})^{-1}\mathcal{P}^{N^{\bot}}\big(\mathcal{O}_{2}
-\mathcal{O}'_{2}+\Psi\big)\mathcal{P}^{N}\Big)(0,0)
\\=&
\Bigg\{\frac{1}{4\pi}\textup{Tr}[R^{T^{(1,0)}X}]\big(\frac{\partial}
{\partial \xi_{i}},\frac{\partial}
{\partial \overline{\xi}_{i}}\big)-\frac{1}{2\pi}\Big\langle R^{B}
(\frac{\partial} {\partial \xi_{i}},\frac{\partial}
{\partial \overline{\xi}_{i}})\frac{\partial}
{\partial \xi_{j}}, \frac{\partial}{\partial\overline{\xi}_{j}}\Big\rangle
\nonumber \\&
+\frac{\sqrt{-1}}{2\pi}\Big\langle \Big(2(\nabla^{B}\nabla^{B}
{\bf J})_{(\frac{\partial}{\partial \xi_{i}}, \frac{\partial}
{\partial \overline{\xi}_{i}})}-\big[R^{B}(\frac{\partial}{\partial \xi_{i}}, \frac{\partial}
{\partial \overline{\xi}_{i}}), {\bf J}\big]\Big)\frac{\partial}
{\partial \xi_{j}}, \frac{\partial}{\partial \overline{\xi}_{j}} \Big\rangle
\nonumber \\&
-\sum_{j=1}^{q}\sum_{k=q+1}^{n}\bigg[ \frac{1}{6\pi}\Big\langle R^{B}
(\frac{\partial}
{\partial \xi_{i}}, \frac{\partial}
{\partial \overline{\xi}_{i}})\frac{\partial}
{\partial \overline{\xi}_{j}}, \frac{\partial}{\partial\overline{\xi}_{k}}\Big\rangle+\frac{1}{2\pi}
\Big(R^{E}+\frac{1}{2}\textup{Tr}[R^{T^{(1,0)}X}]\Big)
\big(\frac{\partial}{\partial \overline{\xi}_{j}}, \frac{\partial}
{\partial \overline{\xi}_{k}}\big)
\nonumber \\&
-\frac{\sqrt{-1}}{6\pi}
\Big\langle \Big(2\big(\nabla^{B}\nabla^{B}{\bf J}\big)_{(\frac{\partial}
{\partial \xi_{i}}, \frac{\partial}
{\partial \overline{\xi}_{i}})}-\big[R^{B}(\frac{\partial}{\partial \xi_{i}}, \frac{\partial}
{\partial \overline{\xi}_{i}}), {\bf J}\big]\Big)\frac{\partial}
{\partial \overline{\xi}_{j}}, \frac{\partial}
{\partial \overline{\xi}_{k}}\Big\rangle \bigg]d\overline{\xi}_{k}\wedge
i_{\frac{\partial}{\partial \xi_{j}}}
\Bigg\}I_{\textup{det}{\overline W}^{\ast}\otimes E}.\nonumber
\end{align}

From (\ref{2.98}), we get
\begin{align}\label{2.99}
&-\Big((\mathcal{L}^{0}_{2})^{-1}\mathcal{P}^{N^{\bot}}\big(\mathcal{O}_{2}
-\mathcal{O}'_{2}+\Psi\big)\mathcal{P}^{N}\Big)(0,0)
\\=&
\Bigg\{\frac{1}{4\pi}\textup{Tr}[R^{T^{(1,0)}X}]\big(\frac{\partial}
{\partial \xi_{i}},\frac{\partial}
{\partial \overline{\xi}_{i}}\big)
-\frac{1}{2\pi}\sum_{j=1}^{n}\bigg\langle R^{B}
(\frac{\partial} {\partial \xi_{i}},\frac{\partial}
{\partial \overline{\xi}_{i}})\frac{\partial}
{\partial \xi_{j}}, \frac{\partial} {\partial \overline{\xi}_{j}}\Big\rangle
\nonumber \\&
+\frac{\sqrt{-1}}{\pi}\sum_{j=1}^{n}\Big\langle\big(\nabla^{B}\nabla^{B}
{\bf J}\big)_{(\frac{\partial}{\partial \xi_{i}}, \frac{\partial}
{\partial \overline{\xi}_{i}})}\frac{\partial}
{\partial \xi_{j}}, \frac{\partial}{\partial \overline{\xi}_{j}} \bigg\rangle
\nonumber \\&
-\frac{1}{2\pi}\sum_{j=1}^{q}\sum_{k=q+1}^{n}\bigg[\Big\langle R^{B}(\frac{\partial}
{\partial \xi_{i}}, \frac{\partial}
{\partial \overline{\xi}_{i}})\frac{\partial}
{\partial \overline{\xi}_{j}}, \frac{\partial}
{\partial \overline{\xi}_{k}}\Big\rangle
\nonumber \\& \ \ \ \  \ \ \ \ \ \ \ \ \ \ \ \ \ \ \ \ \
-\frac{2\sqrt{-1}}{3}\Big\langle \big(\nabla^{B}\nabla^{B}{\bf J}\big)_{
(\frac{\partial}{\partial \xi_{i}}, \frac{\partial}{\partial \overline{\xi}_{i}})}\frac{\partial}{\partial \overline{\xi}_{j}},
\frac{\partial}{\partial \overline{\xi}_{k}} \Big\rangle
\nonumber \\& \ \ \ \  \ \ \ \ \ \ \ \ \ \ \ \ \ \ \ \ \
+\Big(R^{E}+\frac{1}{2}\textup{Tr}\big[R^{T^{(1,0)}X}\big]\Big)
\big(\frac{\partial}{\partial \overline{\xi}_{j}}, \frac{\partial}
{\partial \overline{\xi}_{k}}\big)\bigg]d\overline{\xi}_{k}\wedge
i_{\frac{\partial}{\partial \xi_{j}}}
\Bigg\}I_{\textup{det}{\overline W}^{\ast}\otimes E}.\nonumber
\end{align}
By (\ref{1.15a}) we obtain
\begin{align}\label{1.15b}
\sqrt{-1}\sum_{j=1}^{n}\Big\langle\big(\nabla^{B}\nabla^{B}
{\bf J}\big)_{(\frac{\partial}{\partial \xi_{i}}, \frac{\partial}
{\partial \overline{\xi}_{i}})}\frac{\partial}
{\partial \xi_{j}}, \frac{\partial}{\partial \overline{\xi}_{j}} \bigg\rangle
=\frac{1}{16}\big|\nabla^{B}{\bf J}\big|^{2}.
\end{align}
Substituting (\ref{1.15b}) into (\ref{2.99}) yields
\begin{align}\label{2.100a}
&-\Big((\mathcal{L}^{0}_{2})^{-1}\mathcal{P}^{N^{\bot}}\big(\mathcal{O}_{2}
-\mathcal{O}'_{2}+\Psi\big)\mathcal{P}^{N}\Big)(0,0)
 \\=&
\bigg[\frac{1}{16\pi}\big|\nabla^{B}{\bf J}\big|^{2}
-\frac{1}{2\pi}\Big\langle
R^{B}(\frac{\partial}
{\partial \xi_{i}}, \frac{\partial}
{\partial  \overline{\xi}_{i}})\frac{\partial}
{\partial \xi_{j}}, \frac{\partial}
{\partial \overline{\xi}_{j}}\Big\rangle
\nonumber \\&
+\frac{1}{4\pi}\textup{Tr}
\big[R^{T^{(1,0)}X}\big](\frac{\partial}{\partial \xi_{i}},
\frac{\partial}{\partial\overline{\xi}_{i}})\bigg]I_{\textup{det}{\overline W}^{\ast}\otimes E}
\nonumber \\&-
\frac{1}{2\pi}\sum_{j=1}^{q}\sum_{k=q+1}^{n}\bigg[\Big\langle R^{B}(\frac{\partial}
{\partial \xi_{i}}, \frac{\partial}
{\partial \overline{\xi}_{i}})\frac{\partial}
{\partial \overline{\xi}_{j}}, \frac{\partial}
{\partial \overline{\xi}_{k}}\Big\rangle
\nonumber \\& \ \ \ \  \ \ \ \ \ \ \ \ \ \
-\frac{2\sqrt{-1}}{3}\Big\langle \big(\nabla^{B}\nabla^{B}{\bf J}\big)_{
(\frac{\partial}{\partial \xi_{i}}, \frac{\partial}{\partial \overline{\xi}_{i}})}\frac{\partial}{\partial \overline{\xi}_{j}},
\frac{\partial}{\partial \overline{\xi}_{k}} \Big\rangle
\nonumber \\& \ \ \ \ \ \ \ \ \ \ \ \  \ \
+\Big(R^{E}+\frac{1}{2}\textup{Tr}\big[R^{T^{(1,0)}X}\big]\Big)
\big(\frac{\partial}{\partial \overline{\xi}_{j}}, \frac{\partial}
{\partial \overline{\xi}_{k}}\big)\bigg]d\overline{\xi}_{k}\wedge
i_{\frac{\partial}{\partial \xi_{j}}}I_{\textup{det}(\overline{W}^{\ast})\otimes E}.
\nonumber \end{align}

\noindent
Similar to \cite[(2.21)]{Ma08} we obtain
\begin{align}\label{2.44}
\Big\langle R^{TX}(\frac{\partial}{\partial\xi_{i}},\frac{\partial}{\partial\xi_{j}})
\frac{\partial}{\partial\overline{\xi}_{i}}, \frac{\partial}{\partial\overline{\xi}_{j}}
\Big\rangle=
\frac{1}{32}\big|\nabla^{X}{\bf J}\big|^{2}.
\end{align}
\noindent The difference from \cite[(2.21)]{Ma08} is again the use of coordinates $(\xi_{1}, \ldots, \xi_{n})$
adapted to ${\bf J}$.  Combining (\ref{2.63}), (\ref{2.80}), (\ref{2.44}) and (\ref{2.100a})
we get
\begin{align}\label{2.100b}
&-\Big((\mathcal{L}^{0}_{2})^{-1}\mathcal{P}^{N^{\bot}}\mathcal{O}_{2}
\mathcal{P}^{N}\Big)(0,0)
 \\=&
\bigg[\frac{1}{16\pi}\big|\nabla^{B}{\bf J}\big|^{2}-\frac{1}{64\pi}\big|\nabla^{X}{\bf J}\big|^{2}
+\frac{1}{2\pi}R^{E}(\frac{\partial}{\partial \xi_{i}},
\frac{\partial}{\partial\overline{\xi}_{i}})
\nonumber \\&
+\frac{1}{4\pi}\textup{Tr}
\big[R^{T^{(1,0)}X}\big](\frac{\partial}{\partial \xi_{i}},
\frac{\partial}{\partial\overline{\xi}_{i}})-\frac{1}{2\pi}\Big\langle\big(R^{B}-R^{TX}\big)
(\frac{\partial}{\partial \xi_{i}}, \frac{\partial}{\partial \overline{\xi}_{i}})\frac{\partial}{\partial \xi_{j}},
\frac{\partial}{\partial \overline{\xi}_{j}}\Big\rangle \bigg]
I_{\textup{det}{\overline W}^{\ast}\otimes E}
\nonumber \\&-
\sum_{j=1}^{q}\sum_{k=q+1}^{n}\bigg[\frac{1}{2\pi}\Big\langle R^{B}
(\frac{\partial}{\partial \xi_{i}}, \frac{\partial}{\partial \overline{\xi}_{i}})\frac{\partial}
{\partial \overline{\xi}_{j}}, \frac{\partial}
{\partial \overline{\xi}_{k}}\Big\rangle
+\frac{1}{4\pi} \big(dT_{as}\big)(\frac{\partial}{\partial \xi_{i}}, \frac{\partial}{\partial \overline{\xi}_{i}},
\frac{\partial}{\partial \overline{\xi}_{j}}, \frac{\partial}
{\partial \overline{\xi}_{k}})
\nonumber \\& \ \ \ \ \ \ \ \ \ \ \ \  \ \ \
-\frac{\sqrt{-1}}{3\pi}\Big\langle \big(\nabla^{B}\nabla^{B}{\bf J}\big)_{
(\frac{\partial}{\partial \xi_{i}}, \frac{\partial}{\partial \overline{\xi}_{i}})}\frac{\partial}{\partial \overline{\xi}_{j}},
\frac{\partial}{\partial \overline{\xi}_{k}} \Big\rangle
\nonumber \\& \ \ \ \ \ \ \ \ \ \ \ \  \ \ \
+\frac{1}{2\pi}
\Big(R^{E}+\frac{1}{2}\textup{Tr}\big[R^{T^{(1,0)}X}\big]\Big)
\big(\frac{\partial}{\partial \overline{\xi}_{j}}, \frac{\partial}
{\partial \overline{\xi}_{k}}\big)\bigg]d\overline{\xi}_{k}\wedge
i_{\frac{\partial}{\partial \xi_{j}}}I_{\textup{det}(\overline{W}^{\ast})\otimes E}
\nonumber \\&
+\frac{1}{16\pi}(dT_{as})(\frac{\partial}{\partial \overline{\xi}_{i}},
\frac{\partial}{\partial \overline{\xi}_{j}}, \frac{\partial}{\partial \overline{\xi}_{k}},
\frac{\partial}{\partial \overline{\xi}_{l}})d\overline{\xi}_{k}\wedge d\overline{\xi}_{l}\wedge i_{\frac{\partial}{\partial\xi_{i}}}i_{\frac{\partial}{\partial\xi_{j}}}I_{\textup{det}(\overline{W}^{\ast})\otimes E}.
\nonumber
\end{align}
By (\ref{3.15}) we obtain
\begin{align}\label{2.100c}
&-\Big((\mathcal{L}^{0}_{2})^{-1}\mathcal{P}^{N^{\bot}}\mathcal{O}_{2}
\mathcal{P}^{N}\Big)(0,0)
\nonumber \\=&
\bigg[\frac{1}{2\pi}R^{E}(\frac{\partial}{\partial \xi_{i}},
\frac{\partial}{\partial\overline{\xi}_{i}})+\frac{1}{4\pi}\textup{Tr}
\big[R^{T^{(1,0)}X}\big](\frac{\partial}{\partial \xi_{i}},
\frac{\partial}{\partial\overline{\xi}_{i}})-\frac{1}{32\pi}\Lambda_{\omega}\big(d(\Lambda_{\omega}T_{as})\big)
\\&
+\frac{3}{64\pi}\big|\nabla^{B}{\bf J}\big|^{2}
+\frac{2}{\pi}\Big|\big\langle S^{B}(\frac{\partial}{\partial \overline{\xi}_{i}})\frac{\partial}{\partial \xi_{j}},
\frac{\partial}{\partial \xi_{k}}\big\rangle \Big|^{2}\bigg]
I_{\textup{det}{\overline W}^{\ast}\otimes E}
\nonumber \\&-
\frac{1}{2\pi}\sum_{j=1}^{q}\sum_{k=q+1}^{n}\Big[P(\frac{\partial}{\partial \overline{\xi}_{j}},
\frac{\partial}{\partial \overline{\xi}_{k}})
+R^{E}(\frac{\partial}{\partial \overline{\xi}_{j}}, \frac{\partial}{\partial \overline{\xi}_{k}})
\nonumber \\& \ \ \ \ \ \ \ \ \ \ \ \ \ \ \ \ \ \
-\frac{2\sqrt{-1}}{3}\big\langle (\nabla^{B}\nabla^{B}{\bf J})_{(\frac{\partial}{\partial \xi_{i}}, \frac{\partial}{\partial \overline{\xi}_{i}})}\frac{\partial}{\partial \overline{\xi}_{j}}, \frac{\partial}{\partial \overline{\xi}_{k}} \big\rangle \Big]d\overline{\xi}_{k}\wedge
i_{\frac{\partial}{\partial \xi_{j}}}I_{\textup{det}(\overline{W}^{\ast})\otimes E}
\nonumber \\&
+\frac{1}{16\pi}(dT_{as})(\frac{\partial}{\partial \overline{\xi}_{i}},
\frac{\partial}{\partial \overline{\xi}_{j}}, \frac{\partial}{\partial \overline{\xi}_{k}},
\frac{\partial}{\partial \overline{\xi}_{l}})d\overline{\xi}_{k}\wedge d\overline{\xi}_{l}\wedge i_{\frac{\partial}{\partial\xi_{i}}}i_{\frac{\partial}{\partial\xi_{j}}}I_{\textup{det}(\overline{W}^{\ast})\otimes E}.
\nonumber
\end{align}
Now our main result (\ref{2.1}) follows immediately from (\ref{2.31b}), (\ref{2.31c}),
(\ref{2.53}), (\ref{2.54}), (\ref{2.60}) and (\ref{2.100c}).
The trace formula (\ref{2.1b}) follows from (\ref{2.1}) and the definition of projection
$I_{\textup{det}(\overline{W}^{\ast})\otimes E}$.
This completes the proof of Theorem \ref{t2.1}.

\begin{proof}[Proof of Corollary \ref{t1.2}]
Since $(X, g^{TX}, J)$ is K\"ahler, then the torsion $T$ vanishes, hence
\begin{align}\label{7.7}
T_{as}=0\ \textup{and}\ R^{B}=R^{TX}.
\end{align}
From (\ref{1.14c}), (\ref{1.16}) and (\ref{2.14}), we obtain
\begin{align}\label{7.5}
\Big\langle\big(\nabla^{X}\nabla^{X}{\bf J}\big)_{(\overline{u}_{i}, u_{i})}\overline{u}_{j}, \overline{u}_{k}\Big\rangle=0.
\end{align}
Combining (\ref{1.15}) and (\ref{7.5}) yields
\begin{align}\label{7.6}
\Big\langle\big(\nabla^{X}\nabla^{X}{\bf J}\big)_{(u_{i}, \overline{u}_{i})}\overline{u}_{j}, \overline{u}_{k}\Big\rangle=
-2\sqrt{-1}\Big\langle R^{TX}(u_{i}, \overline{u}_{i})\overline{u}_{j}, \overline{u}_{k}\Big\rangle.
\end{align}
Formula (\ref{1.12e}) follows from (\ref{2.1}), (\ref{7.7}) and (\ref{7.6}). The proof of Corollary \ref{t1.2} is complete.
\end{proof}

\section{Compatibility with Riemann-Roch-Hirzebruch formula} \label{s6}
In this section we check the compatibility of our final result (\ref{2.1}) with
Riemann-Roch-Hirzebruch formulas.

Let us start with the Riemann-Roch-Hirzebruch formula which arises from the
Riemann-Roch-Hirzebruch Theorem (cf. e.g. \cite[Th.\,1.4.6]{Ma07}).
Let $h^{0, q}_{p}$ be the dimension of $H^{0, q}(X, L^{p}$ $\otimes E)$, and let
 $\textup{rk}(E)$ be the rank of $E$. Combining (\ref{1.19}) and
 the Riemann-Roch-Hirzebruch Theorem, we find that
 \begin{align}\begin{split}\label{1.19a}
 (-1)^{q}h^{0, q}_{p}=&\int_{X}\textup{Td}\big(T^{(1, 0)}X\big)\textup{ch}(L^{p}\otimes E)
 \\=&\ \textup{rk}(E)\int_{X}\frac{c_{1}(L)^{n}}{n!}p^{n}+\int_{X}
 \Big(c_{1}(E)+\frac{\textup{rk}(E)}{2}c_{1}\big(T^{(1, 0)}X\big)\Big)\frac{c_{1}(L)^{n-1}}{(n-1)!}p^{n-1}
 \\ &+O(p^{n-2}),
 \end{split}\end{align}
 where $\textup{ch}(\cdot), c_{1}(\cdot), \textup{Td}(\cdot)$ are the Chern character, the first
 Chern class and the Todd class of the corresponding complex vector bundles, respectively.

 By integrating over $X$ the expansion (\ref{1.22}) for $k=1$, we have
\begin{align}\begin{split}\label{2.1e}
&\int_{X}\textup{Tr}\big[P_{p}^{0, q}(x, x)\big]dv_{X}(x)
\\=&
p^{n}\int_{X}\textup{Tr}\big[{\bf b}_{0}(x)\big]dv_{X}(x)
+p^{n-1}\int_{X}\textup{Tr}\big[{\bf b}_{1}(x)\big]dv_{X}(x)+O(p^{n-2}),
\end{split}\end{align}
where the trace is taken over $\Lambda^{q}(T^{\ast(0, 1)}X)\otimes E$. By (\ref{7.1}), we obtain
\begin{align}\label{2.1a}
dv_{X}=\Theta^{n}/n!=(-1)^{q} {\omega}^{n}/n!.
\end{align}
It follows from (\ref{3.2}) that the following identity holds for any smooth $2$-form $\alpha$,
\begin{align}\label{2.1c}
\alpha\wedge \omega^{n-1}/(n-1)!=-\sqrt{-1}\alpha(u_{j}, \overline{u}_{j})\cdot {\omega}^{n}/n!.
\end{align}
Applying (\ref{2.1c}) for $\alpha=d(\Lambda_{\omega}T_{as})$ and the Stokes' Theorem, we obtain
\begin{align}\label{2.1d}
\int_{X}\Lambda_{\omega}\big(d(\Lambda_{\omega}T_{as})\big)dv_{X}
=(-1)^{q}/(n-1)! \cdot \int_{X}d(\Lambda_{\omega}T_{as})\wedge {\omega}^{n-1}=0.
\end{align}
Substituting (\ref{2.1b}), (\ref{2.1a}), (\ref{2.1d}) and the equality (\ref{2.1c}) for $\alpha=c_{1}(E)$ and $c_{1}(T^{(1, 0)}X)$, respectively, into (\ref{2.1e}),  we obtain (\ref{1.19a}).
 Therefore, our final formula (\ref{2.1}) is compatible with (\ref{1.19a}).

On the other hand we also explain here the compatibility of our formula (\ref{2.1})
with the local index formula obtained by Bismut \cite[(2.53)]{Bismut89} for non K\"ahler manifolds under the
assumption that the form $T_{as}$ is closed.

Recall that $S^{B}$ is defined in (\ref{1.11a}). Set
\begin{align}\label{1.1.16}
\nabla^{-B}=\nabla^{TX}+S^{-B}\ \  \textup{with}\ S^{-B}=-S^{B}.
\end{align}
We denote by $R^{-B}$ the curvature of the connection $\nabla^{-B}$.
Note that by (\ref{1.11a}) and \cite[(2.36)]{Bismut89} our notations $S^{B}, R^{-B}$
correspond to $S^{-B}$ and $R^{B}$ in \cite[\S II\,b)]{Bismut89} respectively.
Let $\widehat{A}$ be the Hirzebruch polynomial on $(2n, 2n)$ matrices. Then
\begin{align}\label{1.1.12}
\widehat{A}\big(\frac{R^{-B}}{2\pi}\big)\in \oplus_{j}\ \Omega^{4j}(X),
\end{align}
where $\Omega^{j}(X)$ denotes the space of smooth $j$-forms over $X$.

For $t>0$, let $Q_{p, t}(x, y)$ be the smooth kernel on $X$ associated to the operator
$\textup{exp}(-tD^{2}_{p})$. Let $\Omega^{0,\textup{even}}(X, L^{p}\otimes E)$
(resp. $\Omega^{0,\textup{odd}}(X, L^{p}\otimes E)$) be the direct sum of
the space of smooth $(0, 2j)$-forms (resp. $(0, 2j+1)$-forms) over $X$ with values
in $L^{p}\otimes E$ for $j\geqslant 0$. Set
\begin{align}
\textup{Tr}_{s}=\textup{Tr}\big|_{\Omega^{0,\textup{even}}(X, L^{p}\otimes E)}
-\textup{Tr}\big|_{\Omega^{0,\textup{odd}}(X, L^{p}\otimes E)}.
\end{align}

\noindent
Note that the auxiliary vector bundle $\xi$ in \cite[Th.\,2.11]{Bismut89} should read as $L^{p}\otimes E$.
Denote by $R^{L^{p}\otimes E}$ the curvature of the Chern connection $\nabla^{L^{p}\otimes E}$
on $L^{p}\otimes E$. Then we can restate \cite[Th.\,2.11]{Bismut89} as follows.

\begin{thm}
Assume that $dT_{as}=0$, then
\begin{align}\begin{split}\label{1.1.13}
&\lim_{t\rightarrow 0}\textup{Tr}_{s}\big[Q_{p, t}(x, x)\big]dv_{X}(x)
\\=&
\bigg\{\widehat{A}\big(\frac{R^{-B}}{2\pi}\big)
\textup{exp}\Big(\frac{\sqrt{-1}}{4\pi}\textup{Tr}\big[R^{T^{(1,0)}X}\big]\Big)
\textup{Tr}\Big[\textup{exp}\big(\frac{\sqrt{-1}}{2\pi}R^{L^{p}\otimes E}\big)\Big]\bigg\}^{\textup{max}}
\end{split}\end{align}
uniformly on $X$.
\end{thm}

Now we check the compatibility of our final result (\ref{2.1}) with (\ref{1.1.13}).

Mckean-Singer formula \cite[Th.\,3.50]{Berline04} also holds for the modified Dirac operator $D_{p}$:
\begin{align}\label{7.4}
\sum_{j=0}^{n}(-1)^{j}\ \dim H^{0, j}(X, L^{p}\otimes E)=\int_{X}\textup{Tr}_{s}\big[Q_{p,t}(x, x)\big]dv_{X}(x).
\end{align}

\noindent
Combining (\ref{1.19}), (\ref{1.1.13}) and (\ref{7.4}) yields
\begin{align}\label{1.1.15}
(-1)^{q}h^{0, q}_{p}=\int_{X}\widehat{A}\big(\frac{R^{-B}}{2\pi}\big)
\textup{exp}\big(\frac{1}{2}c_{1}(T^{(1,0)}X)\big)\textup{ch}(L^{p}\otimes E).
\end{align}

\noindent
If we expand the right hand side of the formula (\ref{1.1.15})
in a polynomial of degree $n$ in $p$, then it follows from (\ref{1.1.12}) that the term
$\widehat{A}(\frac{R^{-B}}{2\pi})$ has no contribution to the coefficients of $p^n$ and $p^{n-1}$.
Hence we obtain from (\ref{1.1.15}) that
\begin{align}\begin{split}\label{1.1.14}
(-1)^{q}h^{0, q}_{p}
=&\ \textup{rk}(E)\int_{X}\frac{c_{1}(L)^{n}}{n!}p^{n}+\int_{X}
 \Big(c_{1}(E)+\frac{\textup{rk}(E)}{2}c_{1}\big(T^{(1, 0)}X\big)\Big)\frac{c_{1}(L)^{n-1}}{(n-1)!}p^{n-1}
\\& +O(p^{n-2}),
\end{split}\end{align}
which is exactly the same expression as (\ref{1.19a}).
In this case our formula (\ref{2.1}) is also compatible with (\ref{1.1.13}).
This fits well the compatibility of the asymptotic expansion of Bergman kernel and the
local index theorem along the lines of \cite[Rem.\,4.1.4]{Ma07}, \cite[\S5.1]{Ma12}.

\section{A simple example} \label{s5}
In this section we provide an example in the K\"ahler case.

Let us first consider the 1-dimensional projective space $(\mathbb{C}\mathbb{P}^{1}, J)$
with the complex structure $J$. Let $(\mathcal{O}(-1), h^{\mathcal{O}(-1)})$ be the
tautological line bundle over $\mathbb{C}\mathbb{P}^{1}$, and
let $(\mathcal{O}(1), h^{\mathcal{O}(1)})$ be the dual of the line bundle $\mathcal{O}(-1)$,
i.e., $\mathcal{O}(1)=\mathcal{O}(-1)^{\ast}$. If we denote by
$R^{\mathcal{O}(1)}$ the curvature of the Chern connection on $(\mathcal{O}(1), h^{\mathcal{O}(1)})$
and by $\omega_{FS}$ the Fubini-Study form on $\mathbb{C}\mathbb{P}^{1}$, then
\begin{align}\label{1.1.1}
\omega_{FS}=\frac{\sqrt{-1}}{2\pi}R^{\mathcal{O}(1)}=c_{1}\big(\mathcal{O}(1)\big).
\end{align}
Let $T^{(1, 0)}\mathbb{C}\mathbb{P}^{1}$ be the holomorphic subbundle of the bundle
$T\mathbb{C}\mathbb{P}^{1}\otimes_{\mathbb{R}}\mathbb{C}$. It is straightforward to
verify the following formula
\begin{align}\label{1.1.2}
\textup{ch}(T^{(1, 0)}\mathbb{C}\mathbb{P}^{1})=\big[1+c_{1}\big(\mathcal{O}(1)\big)\big]^{2},
\end{align}
which implies immediately
\begin{align}\label{1.1.3}
c_{1}\big(T^{(1, 0)}\mathbb{C}\mathbb{P}^{1}\big)=2c_{1}\big(\mathcal{O}(1)\big).
\end{align}

Let $(L, h^{L})$ be a holomorphic line bundle over $\mathbb{C}\mathbb{P}^{1}$, and let
$R^{L}$ be the curvature of the Chern connection on $(L, h^{L})$.
Denote by $\omega:=\frac{\sqrt{-1}}{2\pi}R^{L}$ the symplectic form on $\mathbb{C}\mathbb{P}^{1}$.
Clearly,
\begin{align} \omega= \begin{cases} \label{1.1.4}
 -\omega_{FS}, & \mbox{if}\ (L, h^{L})=(\mathcal{O}(-1), h^{\mathcal{O}(-1)}); \\
 \omega_{FS}, & \mbox{if}\ (L, h^{L})=(\mathcal{O}(1), h^{\mathcal{O}(1)}). \end{cases}\end{align}

\noindent
Now combining (\ref{1.1.1}), (\ref{1.1.2}) and (\ref{1.1.4}) we obtain
\begin{align}\label{1.1.5}
\Lambda_{\omega}c_{1}\big(T^{(1, 0)}\mathbb{C}\mathbb{P}^{1}\big)=
\begin{cases}
-2, & \mbox{if}\ (L, h^{L})=(\mathcal{O}(-1), h^{\mathcal{O}(-1)});
\\
2, & \mbox{if}\ (L, h^{L})=(\mathcal{O}(1), h^{\mathcal{O}(1)}).
\end{cases}\end{align}

Set
$X=\mathbb{C}\mathbb{P}^{1}\times \ldots \times \mathbb{C}\mathbb{P}^{1}$ with $n$ copies
and
$L=L_{1}\boxtimes\ldots\boxtimes L_{q}\boxtimes L_{q+1}\boxtimes\ldots\boxtimes L_{n}$
with $L_{k}=\ \mathcal{O}(-1)$ if $k\leqslant q$ and
$L_{k}=\ \mathcal{O}(1)$ otherwise.
Let $h^{L_{k}}$ be the Hermitian metric on the $k$-component $L_{k}$ of the line bundle $L$, i.e.,
$h^{L_{k}}=h^{\mathcal{O}(-1)}$ if $k\leqslant q$ and $h^{L_{k}}=h^{\mathcal{O}(1)}$ otherwise,
and let $R^{L_{k}}$ be the curvature of the Chern connection on $(L_{k}, h^{L_{k}})$.
Denote by $\omega_{k}:=\frac{\sqrt{-1}}{2\pi}R^{L_{k}}$ the symplectic form on the $k$-th component $L_{k}$. Then
\begin{align}\label{1.1.6} \omega_{k}=\begin{cases}
-\omega_{FS}, & \mbox{if}\ 1\leqslant k\leqslant q;
\\
\omega_{FS}, & \mbox{if}\ k\geqslant q+1.
\end{cases}\end{align}
It is clear that the symplectic form $\omega_{L}:=\frac{\sqrt{-1}}{2\pi}R^{L}$ is now just the direct sum of the $\omega_{k}$'s, i.e.,
$\omega_{L}=\omega_{1}+\ldots+\omega_{n}$. Therefore the metric $g^{TX}$ on $X$ defined by (\ref{1.12a})
is now given by the one induced by the Fubini-Study metric $g^{T\mathbb{C}\mathbb{P}^{1}}$ on each factor $\mathbb{C}\mathbb{P}^{1}$ of $X$.
Moreover the skew-adjoint linear map ${\bf J}$ given by (\ref{1.12}) is preserved by the Levi-Civita connection $\nabla^{TX}$ on $(X, g^{TX})$,
i.e.,
\begin{align}\label{1.1.7}
\nabla^{X}{\bf J}=0.
\end{align}
Clearly,
\begin{align}\label{1.1.8}
\frac{\sqrt{-1}}{4}\textup{Tr}\big[\Lambda_{\omega_{L}}R^{T^{(1, 0)}X}\big]=
\frac{\pi}{2}\sum^{n}_{k=1}\Lambda_{\omega_{k}}c_{1}\big(T^{(1, 0)}\mathbb{C}\mathbb{P}^{1}\big)
\end{align}
Combining (\ref{1.1.4}), (\ref{1.1.5}), (\ref{1.1.6}) and (\ref{1.1.8}), we obtain
\begin{align}\label{1.1.9}
\frac{\sqrt{-1}}{4}\textup{Tr}\big[\Lambda_{\omega_{L}}R^{T^{(1, 0)}X}\big]=
\frac{\pi}{2}\big[-2q+2(n-q)\big]=\pi(n-2q).
\end{align}
Substituting (\ref{1.1.7}) and (\ref{1.1.9}) into our main result (\ref{1.12e}) (for K\"ahler manifolds) yields
\begin{align}\label{1.1.10}
{\bf b}_{1}=(n-2q)\cdot I_{\det(\overline{W}^{\ast})},
\end{align}
where the subbundle $W$ is now by definition, the direct sum $T^{(1, 0)}\mathbb{C}\mathbb{P}^{1}\oplus\ldots \oplus T^{(1, 0)}\mathbb{C}\mathbb{P}^{1}$ with $q$ copies over the first $q$-factors of $X$. In particular, if $q=0$,
then the formula (\ref{1.1.10}) reduces to
\begin{align}\label{1.1.11}
{\bf b}_{1}=n\cdot I_{\mathbb{C}}.
\end{align}
Note that  (\ref{1.12f}) and (\ref{1.1.11}) imply the well-known fact, that
the 1-dimensional projective space $(\mathbb{C}\mathbb{P}^{1}, g^{T\mathbb{C}\mathbb{P}^{1}})$ endowed with
the Fubini-Study metric $g^{T\mathbb{C}\mathbb{P}^{1}}$ has constant scalar curvature $8\pi$.

\section{Application to covering manifolds and homogeneous line bundles}
Let $(\tilde{X}, \tilde{J})$ be a paracompact complex manifold of dimension $n$, and
let $\Gamma$ be a discrete group acting holomorphically, freely and
properly discontinuously on $\tilde{X}$ such that the quotient $X=\tilde{X}/\Gamma$ is compact. Assume that there
exists a $\Gamma$-invariant holomorphic Hermitian line bundle $(\tilde{L}, h^{\tilde{L}})$ on $\tilde{X}$, such that
$\tilde{\omega}:=\frac{\sqrt{-1}}{2\pi}R^{\tilde{L}}$ is a symplectic form. The signature of the curvature $R^{\tilde{L}}$
(i.e., the pair of the numbers of negative and positive eigenvalues at a point) with respect to
any Riemannian metric on $T\tilde{X}$, compatible with $\tilde{J}$ is locally constant. We assume that the signature
is constant and equals $(q, n-q)$, $0\leqslant q\leqslant n$.

Let $g^{T\tilde{X}}$ be any $\Gamma$-invariant Riemannian metric on $T\tilde{X}$ compatible with $\tilde{J}$. Consider a $\Gamma$-invariant
holomorphic Hermitian vector bundle $(\tilde{E}, h^{\tilde{E}})$ on $\tilde{X}$. Let  $\Omega^{0, j}(\tilde{X}, \tilde{L}^{p}\otimes \tilde{E})$,
$\Omega^{0, \bullet}(\tilde{X}, \tilde{L}^{p}\otimes \tilde{E})$ be
the spaces defined as before, and let $\Omega_{c}^{0, \bullet}(\tilde{X}, \tilde{L}^{p}\otimes \tilde{E})$
be the subspace of $\Omega^{0, \bullet}(\tilde{X}, \tilde{L}^{p}\otimes \tilde{E})$ consisting of elements with compact support.
We introduce an $L^{2}$-scalar product
on $\Omega_{c}^{0, \bullet}(\tilde{X}, \tilde{L}^{p}\otimes \tilde{E})$ associated to $h^{\tilde{L}}, h^{\tilde{E}}, dv_{T\tilde{X}}$
as in (\ref{1.2}). Let
$L_{0, \bullet}^{2}(\tilde{X}, \tilde{L}^{p}\otimes \tilde{E})$ be the corresponding $L^{2}$-space. We consider the maximal extension
of $\overline{\partial}^{\tilde{L}^{p}\otimes \tilde{E}}$ (see \cite[(3.1.1)]{Ma07})
\begin{align}\label{8.1}
\Dom(\overline{\partial}^{\tilde{L}^{p}\otimes \tilde{E}})=\big\{u\in L_{0, \bullet}^{2}(\tilde{X}, \tilde{L}^{p}\otimes \tilde{E}),\
\overline{\partial}^{\tilde{L}^{p}\otimes \tilde{E}}u\in L_{0, \bullet}^{2}(\tilde{X}, \tilde{L}^{p}\otimes \tilde{E})\big\},
\end{align}
where $\overline{\partial}^{\tilde{L}^{p}\otimes \tilde{E}}u$ is calculated in the sense of distribution. By
replacing $\overline{\partial}^{\tilde{L}^{p}\otimes \tilde{E}}$ with $\overline{\partial}^{\tilde{L}^{p}\otimes \tilde{E}, \ast}$
in (\ref{8.1}) we obtain the maximal extension of the formal adjoint $\overline{\partial}^{\tilde{L}^{p}\otimes \tilde{E}, \ast}$
of $\overline{\partial}^{\tilde{L}^{p}\otimes \tilde{E}}$. The Hodge-Dolbeault operator is defined by
\begin{align}\begin{split}\label{8.2}
\Dom(\tilde{D}_{p})=&\Dom(\overline{\partial}^{\tilde{L}^{p}\otimes \tilde{E}})\cap \Dom(\overline{\partial}^{\tilde{L}^{p}\otimes \tilde{E}, \ast}),
\\
\tilde{D}_{p}=&\sqrt{2}(\overline{\partial}^{\tilde{L}^{p}\otimes \tilde{E}}+\overline{\partial}^{\tilde{L}^{p}\otimes \tilde{E}, \ast}).
\end{split}\end{align}
Note that $g^{T\tilde{X}}$ is complete, being the pullback of a Riemannian metric on $X$. Then by
the Andreotti-Vesentini Lemma \cite[Lemma\,3.3.1]{Ma07}, the maximal extension of $\overline{\partial}^{\tilde{L}^{p}\otimes \tilde{E}, \ast}$
coincides with the Hilbert space adjoint of the maximal extension of $\overline{\partial}^{\tilde{L}^{p}\otimes \tilde{E}}$.

The space of harmonic forms is defined by
\begin{align}\label{8.3}
\mathcal{H}^{0, \bullet}(\tilde{X}, \tilde{L}^{p}\otimes \tilde{E})=\Ker \tilde{D}^{2}_{p}=\Ker \overline{\partial}^{\tilde{L}^{p}\otimes \tilde{E}}\cap \Ker
\overline{\partial}^{\tilde{L}^{p}\otimes \tilde{E}, \ast}.
\end{align}

\begin{thm}[Andreotti-Grauert $L^{2}$ vanishing theorem]\label{t8.1}
In the conditions as above we have
\begin{align}\label{8.4}
\mathcal{H}^{0, j}(\tilde{X}, \tilde{L}^{p}\otimes \tilde{E})=0,\ \textup{for}\ j\neq q,\  p\gg 1.
\end{align}
\end{thm}
\begin{proof}
By proceeding as in \cite[Th.\,1.5]{Ma06} we obtain that there exists $C>0$ such that
for $j\neq q$,
\begin{align}\label{8.5}
\big\|\tilde{D}_{p}s\big\|^{2}\geqslant (2p\mu_{0}-C)\big\|s\big\|^{2},\ \ s\in \Dom(\tilde{D}_{p})\cap
\Omega^{0, j}(\tilde{X}, \tilde{L}^{p}\otimes \tilde{E}),
\end{align}
where $\mu_{0}$ denotes the infimum over $X$ of the absolute values of the eigenfunctions of $R^{L}$. This estimate immediately
implies that
\begin{align}
\Ker \tilde{D}_{p}\cap \Omega^{0, j}(\tilde{X}, \tilde{L}^{p}\otimes \tilde{E})=\{0\},\ \textup{for}\ j\neq q,\ p\gg 1.
\end{align}
This completes the proof the Theorem \ref{t8.1}.
\end{proof}

Theorem \ref{t8.1} was first proved by Braverman \cite[Cor.\,3.6]{Braverman99}.

Define the Bergman kernel $\tilde{P}_{p}(\cdot, \cdot)$ (resp. $\tilde{P}^{0, q}_{p}(\cdot, \cdot)$) as the kernel of the orthogonal projection
$\tilde{P}_{p}: L_{0, \bullet}^{2}(\tilde{X}, \tilde{L}^{p}\otimes \tilde{E})\rightarrow \Ker \tilde{D}_{p}$
(resp. $\tilde{P}^{0, q}_{p}: L_{0, q}^{2}(\tilde{X}, \tilde{L}^{p}\otimes \tilde{E})\rightarrow \Ker \tilde{D}_{p}$).
By Theorem \ref{t8.1}, the kernel $\tilde{P}^{0, q}_{p}(\cdot, \cdot)$ coincides with $\tilde{P}_{p}(\cdot, \cdot)$
for $p\gg 1$. Denote by $\tilde{\Theta}$ the K\"ahler form associated to $g^{T\tilde{X}}$.
Then the Bergman kernel $\tilde{P}^{0, q}_{p}(\cdot, \cdot)$ has an asymptotic expansion on compact sets of
$X$ analogue to Theorem \ref{t1.1}.

\begin{thm}\label{t8.2}
There exist smooth coefficients $\tilde{\bf b}_{r}(x)\in \textup{End}\big(\Lambda^{q}(T^{\ast(0,1)}\tilde{X})\otimes \tilde{E}\big)_{x}$,
which are polynomials in $R^{T\tilde{X}}, R^{\tilde{E}} \ (\textup{and}\ d\tilde{\Theta}, R^{\tilde{L}})$ and
their derivatives of order $\leqslant 2r-2\ (\textup{resp}. \leqslant 2r-1, 2r)$ at $x$, such that
for any compact set $K\subset \tilde{X}$ and any $k,l\in \mathbb{N}$, there exists $C_{k,l}>0$ with
\begin{align}\label{8.6a}
\Big|\tilde{P}^{0,q}_{p}(x,x)-\sum^{k}_{r=0}\tilde{\bf b}_{r}(x)p^{n-r} \Big|_{C^{l}(K)}\leqslant C_{k,l}p^{n-k-1}
\end{align}
for any $p\in \mathbb{N}$.
 \end{thm}
 \begin{proof}
 This follows from \cite[Th.\,6.1.4]{Ma07} and Theorem \ref{t1.1}.
 In fact, let $\pi_{\Gamma}$ be the projection $\tilde{X}\rightarrow X$, and let $P_{p}$ be the
 Bergman kernel on $X=\tilde{X}/\Gamma$. By \cite[(6.1.16)]{Ma07} we obtain
 \begin{align}\label{8.6}
 \tilde{P}_{p}(x, x)-P_{p}(\pi_{\Gamma}(x), \pi_{\Gamma}(x))=O(p^{-\infty}).
 \end{align}
 Then the assertion follows immediately from  (\ref{8.6}) and Theorem \ref{t1.1}.
 \end{proof}

 Choose now the metric $g^{TX}$ as in (\ref{1.12a}).

 \begin{thm}\label{t8.3}
 The coefficient $\tilde{\bf b}_{1}$ of the expansion in Theorem \ref{t8.2} is given by \textup{(\ref{2.1})}.
 \end{thm}
 \begin{proof}
 It follows from (\ref{8.6}) that the Bergman kernel $\tilde{P}_{p}(\cdot, \cdot)$ has the same asymptotic expansion
 as $P_{p}(\pi_{\Gamma}(\cdot), \pi_{\Gamma}(\cdot))$. Hence,
 \begin{align}\label{8.7}
 \tilde{\bf b}_{1}(x)={\bf b}_{1}(\pi_{\Gamma}(x)).
 \end{align}
 Since the metrics on $\tilde{X}, \tilde{L}, \tilde{E}$ are pullbacks of metrics on
 $X, L, E$, then the assertion follows from (\ref{8.7}) and (\ref{2.1}).
 \end{proof}

\begin{rem}
 (i). Let $\textup{Ind}_{\Gamma}(\tilde{D}_{p}), \textup{Ind}(D_{p})$ denote
 the $L^{2}$-index of $\tilde{D}_{p}$ and the index of $D_{p}$, respectively.
 It follows from the $L^{2}$-index Theorem \cite[Th.\,(3.8)]{Atiyah76} that
 \begin{align}\label{8.8}
 \textup{Ind}_{\Gamma}(\tilde{D}_{p})=\textup{Ind}(D_{p}).
 \end{align}
 That is,
 \begin{align}\label{8.10}
 \sum_{j=0}^{n}(-1)^{j}\dim_{\Gamma}\mathcal{H}^{0, j}(\tilde{X}, \tilde{L}^{p}\otimes \tilde{E})
 =\sum_{j=0}^{n}(-1)^{j}\dim H^{0, j}(X, L^{p}\otimes E),
 \end{align}
 where $\dim_{\Gamma}$ denotes the von Neumann dimension.
 Let $h^{0, q}_{p}$ be the dimension of the space $H^{0, q}(X, L^{p}\otimes E)$.
 Combining (\ref{1.19}), (\ref{8.4}) and (\ref{8.10}) yields
 \begin{align}\label{8.8a}
 \dim_{\Gamma}\mathcal{H}^{0, q}(\tilde{X}, \tilde{L}^{p}\otimes \tilde{E})=h^{0, q}_{p},\ \textup{for}\
 p\gg 1.
 \end{align}
 (ii). We denote by $U$ a fundamental domain of $\tilde{X}$ and by $dv_{\tilde{X}}$ the volume form of $(\tilde{X}, g^{T\tilde{X}})$.
 By \cite[(3.6.12)]{Ma07}) and (\ref{8.6a}) we obtain
 \begin{align*}\begin{split}
 \dim_{\Gamma}\mathcal{H}^{0, q}(\tilde{X}, \tilde{L}^{p}\otimes \tilde{E})
 =&\int_{U}\textup{Tr}\big[\tilde{P}^{0, q}_{p}(x, x)\big]dv_{\tilde{X}}(x)
 \\ =&p^n\int_{U} \textup{Tr}\big[\tilde{b}_0(x)\big]dv_{\tilde{X}}(x) +p^{n-1}\int_{U} \textup{Tr}\big[\tilde{b}_1(x)\big]dv_{\tilde{X}}(x)+O(p^{n-2}).
 \end{split}\end{align*}
 \end{rem}

 Let us now consider the case of homogeneous line bundles.
 Let $G$ be a connected noncompact real semi-simple Lie group having a Cartan subgroup $H\subset G$. Let
 $K\supset H$ be a maximal compact subgroup of $G$. Let $\mathfrak{g}, \mathfrak{k}, \mathfrak{h}$ be the
 Lie algebras of $G, K, H$ and $\mathfrak{g}_{\mathbb{C}}, \mathfrak{k}_{\mathbb{C}}, \mathfrak{h}_{\mathbb{C}}$
 their complexifications. Denote by $\Delta$ the set of roots of
 $(\mathfrak{g}_{\mathbb{C}}, \mathfrak{h}_{\mathbb{C}})$ and let $P\subset \Delta$ be a system of positive roots.
 Let $\tilde{X}=G/H$ and let $\tilde{L}_{\lambda}\rightarrow \tilde{X}$ be the line bundle induced on $\tilde{X}$
 by the character $\lambda$ of $H$. By
 \cite[\S 1]{Griffiths69} the choice of the positive root system $P$ defines a complex structure on $\tilde{X}$
 and $\tilde{L}_{\lambda}$,  such that $\tilde{L}_{\lambda}\rightarrow \tilde{X}$ is a holomorphic line bundle.

 By a theorem of Borel \cite{Borel63}, there exists a discrete group $\Gamma\subset G$ which acts freely on $\tilde{X}$
 such that $X=\tilde{X}/\Gamma$ is compact. Moreover, the action of $\Gamma$ lifts to an action on $\tilde{L}_{\lambda}$.
 The following result is due to Griffiths and Schmidt, see \cite[Th.\,4.17D]{Griffiths69}.
 \begin{thm}\label{t8.4}
 Let $\big\langle\cdot, \cdot\big\rangle$ be the scalar product on $\mathfrak{h}^{\ast}$ induced by
 the Cartan-Killing form on $\mathfrak{h}$. Assume that $\lambda$ is a character on $H$ such that
 $\big\langle\lambda, \alpha\big\rangle\neq 0$ for any $\alpha\in \Delta$. Let $\Delta=\Delta_{c}\cup \Delta_{n}$
  be the decomposition of $\Delta$ in compact and noncompact roots with respect to $K$. Set
  \begin{align}
  \iota(\lambda)=\sharp\big\{\alpha\in P\cap \Delta_{c}: \big\langle\lambda, \alpha\big\rangle<0\big\}
  +\sharp \big\{\alpha\in P\cap \Delta_{n}, \big\langle\lambda, \alpha\big\rangle>0\big\}.
  \end{align}
  Let $\rho$ denote the half-sum of the positive roots of $(\mathfrak{g}, \mathfrak{h})$ and
  $q=\iota(\lambda+\rho)$. Then $\tilde{L}_{\lambda}$ admits a $\Gamma$-invariant Hermitian metric $h^{\tilde{L}_{\lambda}}$ whose
  curvature $R^{\tilde{L}_{\lambda}}$ is nondegenerate and has signature $(q, n-q)$.
 \end{thm}

 \noindent Therefore, Theorem \ref{8.3} applies to the calculation of the coefficient $\tilde{\bf b}_{1}$ of the
 Bergman kernel expansion for the homogeneous line bundle $\tilde{L}_{\lambda}$ from Theorem \ref{t8.4}.

\noindent
\textbf{\emph{Acknowledgements.}}
The author would like to thank Professors Xiaonan Ma and George Marinescu for their kind advices.

\end{document}